\documentclass[a4paper,11pt]{article}%
\usepackage{amsfonts}
\usepackage{amssymb}
\usepackage{amsmath}
\usepackage{latexsym}
\usepackage{graphicx}
\usepackage[latin1]{inputenc}
\usepackage{hyperref}
\usepackage{amsthm}
\usepackage{graphicx}
\usepackage{color}%
\usepackage[affil-it]{authblk}

\providecommand{\U}[1]{\protect\rule{.1in}{.1in}}
\font\teneusb=eusb10 \font\seveneusb=eusb7 \font\fiveeusb=eusb5
\newfam\eusbfam \textfont\eusbfam=\teneusb
\scriptfont\eusbfam=\seveneusb \scriptscriptfont\eusbfam=\fiveeusb

\font\tenbifull=cmmib10
\font\tenbimed=cmmib7
\font\tenbismall=cmmib5
\mathchardef\bbGamma="7000 \mathchardef\bbDelta="7001
\mathchardef\bbPhi="7002 \mathchardef\bbAlpha="7003
\mathchardef\bbXi="7004 \mathchardef\bbPi="7005
\mathchardef\bbSigma="7006 \mathchardef\bbUpsilon="7007
\mathchardef\bbTheta="7008 \mathchardef\bbPsi="7009
\mathchardef\bbOmega="700A \mathchardef\bbalpha="710B
\mathchardef\bbbeta="710C \mathchardef\bbgamma="710D
\mathchardef\bbdelta="710E \mathchardef\bbepsilon="710F
\mathchardef\bbzeta="7110 \mathchardef\bbeta="7111
\mathchardef\bbtheta="7112 \mathchardef\bbiota="7113
\mathchardef\bbkappa="7114 \mathchardef\bblambda="7115
\mathchardef\bbmu="7116 \mathchardef\bbnu="7117
\mathchardef\bbxi="7118 \mathchardef\bbpi="7119
\mathchardef\bbrho="711A \mathchardef\bbsigma="711B
\mathchardef\bbtau="711C \mathchardef\bbupsilon="711D
\mathchardef\bbphi="711E \mathchardef\bbchi="711F
\mathchardef\bbpsi="7120 \mathchardef\bbomega="7121
\mathchardef\bbvarepsilon="7122 \mathchardef\bbvartheta="7123
\mathchardef\bbvarpi="7124 \mathchardef\bbvarrho="7125
\mathchardef\bbvarsigma="7126 \mathchardef\bbvarphi="7127

\newcommand{\N}{{\rm I}\kern-0.18em{\rm N}}

\newcommand{\h}{{\rm I}\kern-0.18em{\rm H}}
\newcommand{\K}{{\rm I}\kern-0.18em{\rm K}}

\newcommand{\Z}{{\rm Z}\kern-0.34em{\rm Z}}

\newcommand{\1}{{\rm 1}\kern-0.22em{\rm I}}

\newtheorem{thm}{Theorem}[section]

\newtheorem{ex}{Example}[section]

\newtheorem{cor}{Corollary}[section]

\newtheorem{lem}{Lemma}[section]
\newtheorem{rem}{Remark}[section]

\numberwithin{equation}{section}

\newcounter{eqroman}
\begin{document}

\title{An Abelian theorem with application to the conditional Gibbs principle}
\author{Zhansheng Cao}

\affil{LSTA, Université Paris 6, France}

\maketitle

\begin{abstract}
Let $X_1,...,X_n$ be $n$ independent unbounded real random variables which have common , roughly speaking, light-tailed type distribution. Denote by $S_1^n$ their sum and by $\pi^{a_n}$ the tilted density of $X_1$, where $a_n \rightarrow\infty$ as $n\rightarrow \infty$. An Abelian type theorem is given, which is used to approximate the first three centered moments of the distribution $\pi^{a_n}$. Further, we provide the Edgeworth expansion of $n$-convolution of the normalized tilted density under the setting of a triangular array of row-wise independent summands, which is then applied to obtain one local limit theorem conditioned on extreme deviation event $(S_1^n/n=a_n)$ with $a_n\rightarrow \infty$.

\begin{flushleft}
\end{flushleft}

\end{abstract}

\section{\bigskip Introduction} \label{secintroduction}

It will be assumed that $P_{X}$, which is the distribution of $X_{1}$, has a
density $p$ with respect to the Lebesgue measure on $\mathbb{R}$. The fact
that $X_{1}$ has a light tail is captured in the hypothesis that $X_{1}$ has a
moment generating function
\[
\Phi(t):=E\exp tX_{1}%
\]
which is finite in a non void neighborhood $\mathcal{N}$ of $0.$ This fact is
usually refered to as a Cramer type condition.

Defined on $\mathcal{N}$ are the following functions. The functions
\[
t\rightarrow m(t):=\frac{d}{dt}\log\Phi(t)
\]%
\[
t\rightarrow s^{2}(t):=\frac{d}{dt}m(t)
\]

\[
t\rightarrow\mu_{j}(t):=\frac{d}{dt}s^{2}(t)\text{ \ , \ }j=3,4
\]
are the expectation and the three first centered moments of the r.v.
$\mathcal{X}_{t}$ with density
\[
\pi_{t}(x):=\frac{\exp tx}{\Phi(t)}p(x)
\]
which is defined on $\mathbb{R}$ and which is the tilted density with
parameter $t.$ When $\Phi$ is steep, meaning that
\[
\lim_{t\rightarrow t^{+}}m(t)=\infty
\]
where $t^{+}:=ess\sup\mathcal{N}$ then $m$ parametrizes the convex hull of the
support of $P_{X}.$ We refer to Barndorff-Nielsen \cite{barndorff} for those properties.
As a consequence of this fact, for all $a$ in the support of $P_{X}$, it will
be convenient to define
\[
\pi^{a}=\pi_{t}%
\]
where $a$ is the unique solution of the equation $m(t)=a.$

Let $X_{1}^{n}:=\left(  X_{1},...,X_{n}\right)  $ denote $n$ independent
unbounded real valued random variables and $S_{1}^{n}:=X_{1}+...+X_{n}$ denote
their sum.

 A first contribution of this paper is the approximation of the first three moments of the tilted density, from which we obtain the Edgeworth expansion of $n$-convolution of the normalized tilted density. It is worthwhile to note that this expansion is under the setting of a triangular array of row-wise independent summands.

We now come to some remark on the Gibbs conditional principle in the standard
above setting. A phrasing of this principle is: As $n$ tends to infinity the conditional distribution of $X_{1}$ given
$\left(  S_{1}^{n}/n=a\right)  $ is $\Pi^{a},$ the distribution with density
$\pi^{a}$ (See \cite{DIFR1988} and \cite{DEZE1996}).

Another contribution is that we obtain the local limit distribution of $(X_1,...,X_k)$ for fixed integer $k$ conditioned on extreme deviations (ED) pertaining
to $S_{1}^{n}.$ By extreme deviation we mean that $S_1^n/n$ is supposed to take values which are
going to infinity as $n$ increases (See also Borovkov and Mogul$'$ski$\breve{\i}$ \cite{BorovMogul1} \cite{BorovMogul2}, where the authors call this case \lq\lq superlarge deviation"). It will be showed that for fixed $k$, any set of
r.v.'s $X_{i_1},...,X_{i_k}$ are asymptotically independent given $S_1^n/n=a_n$ with $1\le i_1 \le i_k \le n$, which extends the events by Dembo and Zeitouni \cite{DEZE1996} to the extreme deviation case; see also Csisz\'{a}r \cite{CS1984} for a similar result.

The paper is organized as follows. Notation and hypotheses are stated in
Section \ref{notation}, along with some necessary facts from asymptotic analysis in the
context of light tailed densities. In Section \ref{secAppM}, the approximations of the expectation and the two first centered moments of the tilted density are given. Section \ref{secEdg} states the Edgeworth expansion under extreme normalizing factors. Section \ref{secGibbs} provides the Gibbs' conditional limit theorem under extreme events.

The main tools to be used come from asymptotic analysis and local limit
theorems, developed from \cite{Feller71} and \cite{Bingham}; we also have
borrowed a number of arguments from \cite{Nagaev}. A number of technical
lemmas have been postponed to Section \ref{Chp2proof}. \bigskip

\section{Notation and hypotheses}\label{notation}

In this paper, we consider the uniformly bounded density function $p(x)$
\begin{equation}
p(x)=c\exp\Big(-\big(g(x)-q(x)\big)\Big)\qquad x\in\mathbb{R}_{+}%
,\label{densityFunction}%
\end{equation}
where $c$ is some positive normalized constant. Define $h(x):=g^{\prime}(x) $.
We assume that there exists some positive constant $\vartheta$ ,
for large $x$, it holds
\begin{equation}
\sup_{|v-x|<\vartheta x}|q(v)|\leq\frac{1}{x\sqrt{h(x)}}%
.\label{densityFunction01}%
\end{equation}
The function $g$ is positive and satisfies
\begin{equation}
\frac{g(x)}{x}\longrightarrow\infty,\qquad x\rightarrow\infty
.\label{3section101}%
\end{equation}

Not all positive $g$'s satisfying $(\ref{3section101})$ are adapted to our
purpose. Regular functions $g$ are defined as follows. We define firstly a
subclass $R_{0}$ of the family of \emph{slowly varying} function. A function
$l$ belongs to $R_{0}$ if it can be represented as
\begin{equation}
l(x)=\exp\Big(\int_{1}^{x}\frac{\epsilon(u)}{u}du\Big),\qquad x\geq
1,\label{3section102}%
\end{equation}
where $\epsilon(x)$ is twice differentiable and $\epsilon(x)\rightarrow0$ as
$x\rightarrow\infty$.

We follow the line of Juszczak and Nagaev $\cite{Nagaev}$ to describe the
assumed regularity conditions of $h$.

\textbf{Class ${R_{\beta}}$ :} $h(x)\in{R_{\beta}}$, if, with $\beta>0$ and
$x$ large enough, $h(x)$ can be represented as
\[
h(x)=x^{\beta}l(x),
\]
where $l(x)\in R_{0}$ and in $(\ref{3section102})$ $\epsilon(x)$ satisfies
\begin{equation}
\limsup_{x\rightarrow\infty}x|\epsilon^{\prime}(x)|<\infty,\qquad
\limsup_{x\rightarrow\infty}x^{2}|\epsilon^{^{\prime\prime}}(x)|<\infty
.\label{3section104}%
\end{equation}

\textbf{Class ${R_{\infty}}$ :} Further, $l\in\widetilde{R_{0}}$, if, in
$(\ref{3section102})$, $l(x)\rightarrow\infty$ as $x\rightarrow\infty$ and
\begin{equation}
\lim_{x\rightarrow\infty}\frac{x\epsilon^{\prime}(x)}{\epsilon(x)}%
=0,\qquad\lim_{x\rightarrow\infty}\frac{x^{2}\epsilon^{^{\prime\prime}}%
(x)}{\epsilon(x)}=0,\label{3section103}%
\end{equation}
and for some $\eta\in(0,1/8)$
\begin{equation}
\liminf_{x\rightarrow\infty}x^{\eta}\epsilon(x)>0.\label{3section1030}%
\end{equation}
We say that $h\in{R_{\infty}}$ if $h$ is increasing and strictly monotone and
its inverse function $\psi$ defined through
\begin{equation}
\psi(u):=h^{\leftarrow}(u):=\inf\left\{  x:h(x)\geq u\right\}
\label{inverse de h}%
\end{equation}
belongs to $\widetilde{R_{0}}$.

Denote $\mathfrak{R:}={R_{\beta}}\cup{R_{\infty}}$. In fact, $\mathfrak{R} $
covers one large class of functions, although, ${R_{\beta}}$ and ${R_{\infty}%
}$ are only subsets of \emph{Regularly varying} and \emph{Rapidly varying}
functions, respectively.

\begin{rem}
The role of $(\ref{3section102})$ is to make $h(x)$ smooth enough. Under
$(\ref{3section102})$ the third order derivative of $h(x)$ exists, which is
necessary in order to use a Laplace method for the asymptotic evaluation of
the moment generating function $\Phi(t)$ as $t\rightarrow\infty$, where
\[
\Phi(t)=\int_{0}^{\infty}e^{tx}p(x)dx=c\int_{0}^{\infty}\exp
\Big(K(x,t)+q(x)\Big)dx,\qquad t\in(0,\infty)
\]
in which
\[
K(x,t)=tx-g(x).
\]
If $h\in\mathfrak{R}$, $K(x,t)$ is concave with respect to $x$ and takes its
maximum at $\hat{x}=h^{\leftarrow}(t)$. \ The evaluation of $\Phi(t)$ for
large $t$ follows from an expansion of $K(x,t)$ in a neighborhood of $\hat
{x};$ this is Laplace's method. This expansion yields
\begin{align}\label{abelK(t)}
K(x,t)=K(\hat{x},t)-\frac{1}{2}h^{\prime}(\hat{x})\big(x-\hat{x}%
\big)^{2}-\frac{1}{6}h^{\prime\prime}(\hat{x})\big(x-\hat{x}\big)^{3}%
+\varepsilon(x,t),
\end{align}
where $\varepsilon(x,t)$ is some error term. Conditions $(\ref{3section103})$
$(\ref{3section1030})$ and $(\ref{3section104})$ guarantee that $\varepsilon
(x,t)$ goes to $0$ when $t$ tends to $\infty$ and $x$ belongs to some
neighborhood of $\hat{x}$.
\end{rem}

\begin{ex}
\textbf{Weibull Density. }Let $p$ be a Weibull density with shape parameter
$k>1$ and scale parameter $1$, namely%
\begin{align*}
p(x)  & =kx^{k-1}\exp(-x^{k}),\qquad x\geq0\\
& =k\exp\Big(-\big(x^{k}-(k-1)\log x\big)\Big).
\end{align*}
Take $g(x)=x^{k}-(k-1)\log x$ and $q(x)=0$. Then it holds
\[
h(x)=kx^{k-1}-\frac{k-1}{x}=x^{k-1}\big(k-\frac{k-1}{x^{k}}\big).
\]
Set $l(x)=k-(k-1)/x^{k},x\geq1$, then $(\ref{3section102})$ holds, namely,
\[
l(x)=\exp\Big(\int_{1}^{x}\frac{\epsilon(u)}{u}du\Big),\qquad x\geq1,
\]
with
\[
\epsilon(x)=\frac{k(k-1)}{kx^{k}-(k-1)}.
\]
The function $\epsilon$ is twice differentiable and goes to $0$ as
$x\rightarrow\infty$. Additionally, $\epsilon$ satisfies condition
$(\ref{3section104})$. Hence we have shown that $h\in R_{k-1}$.
\end{ex}

\begin{ex}
\textbf{A rapidly varying density.} Define $p$ through
\[
p(x)=c\exp(-e^{x-1}),\qquad x\geq0.
\]
Then $g(x)=h(x)=e^{x-1}$ and $q(x)=0$ for all non negative $x$. We show that
$h\in R_{\infty}$. It holds $\psi(x)=\log x+1$. Since $h(x)$ is increasing and
monotone, it remains to show that $\psi(x)\in\widetilde{R_{0}}$. When $x\geq
1$, $\psi(x)$ admits the representation of $(\ref{3section102})$ with
$\epsilon(x)=1/(\log x+1)$. Also conditions $(\ref{3section103})$ and
$(\ref{3section1030})$ are satisfied. Thus $h\in R_{\infty}$.
\end{ex}

Throughout the paper we use the following notation. When a r.v. $X$ has
density $p$ we write $p(X=x)$ instead of $p(x).$ This notation is useful when
changing measures. For example $\pi^{a}(X=x)$ is the density at point $x$ for
the variable $X$ generated under $\pi^{a}$, while $p(X=x)$ states for $X$
generated under $p.$ This avoids constant changes of notation.

\section{An Abelian-type theorem}\label{secAppM}
We inherit of the definition of the tilted density $\pi^{a}$ defined in
Section \ref{secintroduction}, and of the corresponding definitions of the functions $m$, $s^{2}$
and $\mu_{3}$. Because of (\ref{densityFunction}) and the various
conditions on $g$ those functions are defined as $t\rightarrow\infty.$ The
following Theorem is basic for the proof of the remaining results.

\begin{thm}
\label{order of s} Let $p(x)$ be defined as in $(\ref{densityFunction})$ and
$h(x)\in\mathfrak{R}$. Denote by
\[
m(t)=\frac{d}{dt}\log\Phi(t),\quad\quad s^{2}(t)=\frac{d}{dt}m(t),\qquad
\mu_{3}(t)=\frac{d^{3}}{dt^{3}}\log\Phi(t),
\]
then with $\psi$ defined as in (\ref{inverse de h}) it holds as $t\rightarrow
\infty$
\[
m(t)\sim\psi(t),\qquad s^{2}(t)\sim\psi^{\prime}(t),\qquad\mu_{3}(t)\sim
\frac{M_{6}-9}{6}\psi^{^{\prime\prime}}(t),
\]
where $M_{6}$ is the sixth order moment of standard normal distribution.
\end{thm}

\begin{proof}[Proof]The proof of this result relies on a series of Lemmas. Lemmas $(\ref{3lemma0}%
)$, $(\ref{3lemma01})$, $(\ref{3lemma02})$ and $(\ref{3lemma1})$ are used in
the proof. Lemma $(\ref{3lemma00})$ is instrumental for Lemma $(\ref{3lemma1}%
)$. The proof of Theorem $\ref{order of s}$ and these Lemmas are postponed to
Section \ref{3appendix01}. \newline
\end{proof}

\begin{rem}
As a by-product of Theorem \ref{order of s}, we obtained the following Abel type result (see $(\ref{3moment001})$):
\begin{align*}
\Phi(t)=c\sqrt{2\pi}\sigma e^{K(\hat{x},t)}\big(1+o(1)\big),
\end{align*}
where $K(\hat{x},t)$ is defined as in $(\ref{abelK(t)})$ and $\sigma$ is defined in Section \ref{3appendix01}. It is easily verified that this result is in accordance with Theorem 4.12.11 of \cite{Bingham}, Theorem 3 of \cite{Borovkov} and Theorem 4.2 of \cite{Nagaev}.
\end{rem}

\begin{cor}
\label{3cor1} Let $p(x)$ be defined as in $(\ref{densityFunction})$ and
$h(x)\in\mathfrak{R}$. Then it holds as $t\rightarrow\infty$
\begin{align}
\label{3cor10g}\frac{\mu_{3}(t)}{s^{3}(t)}\longrightarrow0.
\end{align}

\end{cor}

\begin{proof}[Proof]Its proof relies on Theorem $\ref{order of s}$ and is also put in Section \ref{3appendix01}.
\end{proof}

\section{Edgeworth expansion under extreme normalizing factors}\label{secEdg}

\bigskip
With $\pi^{a_{n}}$ defined through
\begin{align}\label{tiltedDensity}
\pi^{a_{n}}(x)=\frac{e^{tx}p(x)}{\Phi(t)},
\end{align}
and $t$ determined by $m(t)=a_{n}$, define the normalized density of
$\pi^{a_{n}}$ by
\[
\bar{\pi}^{a_{n}}(x)=s\pi^{a_{n}}(sx+a_{n}),
\]
where $s$ is defined in Section \ref{secintroduction} (notice that it depends on $a_n$ here).
Denote the $n$-convolution of $\bar{\pi}^{a_{n}}(x)$ by $\bar{\pi}%
_{n}^{a_{n}}(x)$, and denote by $\rho_{n}$ the normalized density of
$n$-convolution $\bar{\pi}_{n}^{a_{n}}(x)$,
\[
\rho_{n}(x):=\sqrt{n}\bar{\pi}_{n}^{a_{n}}(\sqrt{n}x).
\]
The following result extends the local Edgeworth expansion of the distribution
of normalized sums of i.i.d. r.v.'s to the present context, where the summands
are generated under the density $\bar{\pi}^{a_{n}}$. Therefore the setting is
that of a triangular array of row-wise independent summands; the fact that
$a_{n}\rightarrow\infty$ makes the situation unusual. We mainly adapt Feller's
proof (Chapter 16, Theorem 2 $\cite{Feller71}$).

\begin{thm}
\label{3theorem1} With the above notation, uniformly upon $x$ it holds
\[
\rho_{n}(x)=\phi(x)\Big(1+\frac{\mu_{3}}{6\sqrt{n}s^{3}}\big(x^{3}%
-3x\big)\Big)+o\Big(\frac{1}{\sqrt{n}}\Big).
\]
where $\phi(x)$ is standard normal density.
\end{thm}

\begin{proof}[Proof]
The proof of this Theorem is postponed to
Section \ref{3appendix02}.
\end{proof}

\section{Gibbs' conditional principles under extreme events}\label{secGibbs}

We now explore Gibbs conditional principles under extreme events. The
result obtained is a pointwise approximation of the conditional density $p_{a_{n}%
}\left(  y_{1}^{k}\right)  $ on $\mathbb{R}^{k}$ for fixed $k.$

Fix $y_{1}^{k}:=\left(  y_{1},...,y_{k}\right)  $ in $\mathbb{R}^{k}$ and
define $s_{i}^{j}:=y_{i}+...+y_{j}$ for $1\leq i<j\leq k.$ Define $t$ through $m(t)=a_n$, similarly, define $t_{i}$ through
\begin{equation}
m(t_{i}):=\frac{na_{n}-s_{1}^{i}}{n-i}.\label{3lfd01}%
\end{equation}
For the sake of brevity, we write $m_{i}$ instead of $m(t_{i})$, and define
$s_{i}^{2}:=s^{2}(t_{i})$. Consider the following
condition%
\begin{equation}
\lim_{n\rightarrow\infty}\frac{\psi(t)^{2}}{\sqrt{n}\psi^{\prime}(t)
}=0,\label{croissance de a}%
\end{equation}
which can be seen as a growth condition on $a_{n}$, avoiding too large
increases of this sequence.

For $0\le i \le k-1$, define $z_{i}$ through
\begin{align*}
z_{i}=\frac{m_{i}-y_{i+1}}{s_{i} \sqrt{n-i-1}}.
\end{align*}

\begin{rem}

Formula $(\ref{croissance de a})$ states the precise behaviour of the sequence $a_n$ which defines the present extended Gibbs principle. In the case when the common density $p(x)$ is Weibull with shape parameter $k$, using Theorem $\ref{order of s}$, we obtain $\psi(t)\sim m(t)=a_n$ and $\psi^\prime(t)\sim a_n^{2-k}$. Replace $\psi(t)$ and $\psi^\prime(t)$ in $(\ref{croissance de a})$ by these two terms, we have
\begin{align*}
\lim_{n\rightarrow\infty} \frac{a_n^k}{\sqrt{n}}=0.
\end{align*}
This rate controls the growth of $a_n$ to infinity.

\end{rem}

\begin{lem}
\label{3lemma z} Assume that $p(x)$ satisfies $(\ref{densityFunction})$ and
$h(x)\in\mathfrak{R}$. Let $t_{i}$ be defined in $(\ref{3lfd01})$. Assume that
$a_{n}\rightarrow\infty$ as $n\rightarrow\infty$ and that
(\ref{croissance de a}) holds. Then as ${n}\rightarrow\infty$
\[
\lim_{n\rightarrow\infty}\sup_{0\leq i\leq k-1}z_{i}=0,\qquad \emph{and} \qquad \lim_{n\rightarrow\infty}\sup_{0\leq i\leq k-1}z_{i}^2=o\left(\frac{1}{\sqrt{n}}\right).
\]

\end{lem}

\begin{proof}
The proof of this Lemma is postponed in Section \ref{3appendix03}.
\end{proof}

\begin{thm}
\label{point conditional density} With the same notation and hypotheses as in Lemma \ref{3lemma z}, it holds
\[
p_{a_{n}}(y_{1}^{k})=p(X_{1}^{k}=y_{1}^{k}|S_{1}^{n}=na_{n})=g_{m}(y_{1}%
^{k})\Big(1+o\big(\frac{1}{\sqrt{n}}\big)\Big),
\]
with
\[
g_{m}(y_{1}^{k})=\prod_{i=0}^{k-1}\Big(\pi^{m_{i}}(X_{i+1}=y_{i+1})\Big).
\]

\end{thm}

\begin{proof}Using Bayes formula,
\begin{align}
p_{a_{n}}\left(  y_{1}^{k}\right)  &:=p(X_{1}^{k}=y_{1}^{k}|S_{1}^{n}=na_{n})\nonumber\\
&  =p(X_{1}=y_{1}|S_{1}^{n}=na_{n})\prod_{i=1}^{k-1}p(X_{i+1}=y_{i+1}%
|X_{1}^{i}=y_{1}^{i},S_{1}^{n}=na_{n})\nonumber\label{bayes formula01}\\
&  =\prod_{i=0}^{k-1}p(X_{i+1}=y_{i+1}|S_{i+1}^{n}=na_{n}-s_{1}^{i}).
\end{align}
We make use of the following invariance property: for all $y_{1}^{k}$ and all
$\alpha>0$%
\[
p(X_{i+1}=y_{i+1}|X_{1}^{i}=y_{1}^{i},S_{1}^{n}=na_{n})=\pi^{\alpha}%
(X_{i+1}=y_{i+1}|X_{1}^{i}=y_{1}^{i},S_{1}^{n}=na_{n})
\]
where on the LHS, the r.v's $X_{1}^{i}$ are sampled i.i.d. under $p$ and on
the RHS, sampled i.i.d. under $\pi^{\alpha}.$ It thus holds

\begin{align}
&  p(X_{i+1}=y_{i+1}|S_{i+1}^{n}=na_{n}-S_{1}^{i})=\pi^{m_{i}}(X_{i+1}%
=y_{i+1}|S_{i+1}^{n}=na_{n}-s_{1}^{i})\nonumber\label{bayes formula}\\
&  =\pi^{m_{i}}(X_{i+1}=y_{i+1})\frac{\pi^{m_{i}}(S_{i+2}^{n}=na_{n}%
-s_{1}^{i+1})}{\pi^{m_{i}}(S_{i+1}^{n}=na_{n}-s_{1}^{i})}\nonumber\\
&  =\frac{\sqrt{n-i}}{\sqrt{n-i-1}}\pi^{m_{i}}(X_{i+1}=y_{i+1})\frac
{\widetilde{\pi_{n-i-1}}(\frac{m_{i}-y_{i+1}}{s_{i}\sqrt{n-i-1}})}%
{\widetilde{\pi_{n-i}}(0)},
\end{align}
where $\widetilde{\pi_{n-i-1}}$ is the normalized density of $S_{i+2}^{n}$
under i.i.d. sampling with the density $\pi^{m_{i}};$ correspondingly, $\widetilde
{\pi_{n-i}}$ is the normalized density of $S_{i+1}^{n}$ under the same
sampling. Note that a r.v. with density $\pi^{m_i}$ has expectation $m_{i}$ and
variance $s_{i}^{2}$.

Write $z_{i}=\frac{m_{i}-y_{i+1}}{s_{i}\sqrt{n-i-1}}$, and perform a
third-order Edgeworth expansion of $\widetilde{\pi_{n-i-1}}(z_{i})$, using
Theorem $\ref{3theorem1}$. It follows
\begin{align}\label{edgeworth exte}
\widetilde{\pi_{n-i-1}}(z_{i})=\phi(z_{i})\Big(1+\frac{\mu_{3}^{i}}{6s_{i}%
^{3}\sqrt{n-1}}(z_{i}^{3}-3z_{i})\Big)+o\Big(\frac{1}{\sqrt{n}}\Big),
\end{align}
The approximation of $\widetilde{\pi_{n-i}}(0)$ is obtained from
$(\ref{edgeworth exte})$
\begin{align}\label{pi Z=0}
\widetilde{\pi_{n-i}}(0)=\phi(0)\Big(1+o\big(\frac{1}{\sqrt{n}}\big)\Big).
\end{align}
Put $(\ref{edgeworth exte})$ and $(\ref{pi Z=0})$ into $(\ref{bayes formula})
$ to obtain
\begin{align}
&  p(X_{i+1}=y_{i+1}|S_{i+1}^{n}=na_{n}-S_{1}^{i}%
)\nonumber\label{3conditionaldensity01}\\
&  =\frac{\sqrt{n-i}}{\sqrt{n-i-1}}\pi^{m_{i}}(X_{i+1}=y_{i+1})\frac
{\phi(z_{i})}{\phi(0)}\Big[1+\frac{\mu_{3}^{i}}{6s_{i}^{3}\sqrt{n-1}}%
(z_{i}^{3}-3z_{i})+o\Big(\frac{1}{\sqrt{n}}\Big)\Big]\nonumber\\
&  =\frac{\sqrt{2\pi(n-i)}}{\sqrt{n-i-1}}\pi^{m_{i}}(X_{i+1}=y_{i+1}%
){\phi(z_{i})}\big(1+R_{n}+o(1/\sqrt{n})\big),
\end{align}
where
\[
R_{n}=\frac{\mu_{3}^{i}}{6s_{i}^{3}\sqrt{n-1}}(z_{i}^{3}-3z_{i}).
\]

Under condition $(\ref{croissance de a})$, using Lemma $\ref{3lemma z}$, it holds
$z_{i}\rightarrow0$ as $a_{n}\rightarrow\infty$, and under Corollary
$(\ref{3cor1})$, $\mu_{3}^{i}/s_{i}^{3}\rightarrow0.$ This yields
\[
R_{n}=o\big(1/\sqrt{n}\big),
\]
which, combined with $(\ref{3conditionaldensity01})$, gives
\begin{align*}
&  p(X_{i+1}=y_{i+1}|s_{i+1}^{n}=na_{n}-S_{1}^{i})=\frac{\sqrt{2\pi(n-i)}%
}{\sqrt{n-i-1}}\pi^{m_{i}}(X_{i+1}=y_{i+1}){\phi(z_{i})}\big(1+o(1/\sqrt
{n})\big)\\
&  =\frac{\sqrt{n-i}}{\sqrt{n-i-1}}\pi^{m_{i}}(X_{i+1}=y_{i+1}){\big(1-z_{i}%
^{2}/2+o(z_{i}^{2})\big)}\big(1+o(1/\sqrt{n})\big),
\end{align*}
where we use a Taylor expansion in the second equality. Using once more Lemma
$\ref{3lemma z}$, under conditions $(\ref{croissance de a})$, we have as $a_{n}\rightarrow\infty$
\[
z_{i}^{2}=o(1/\sqrt{n}),
\]
hence we get
\[
p(X_{i+1}=y_{i+1}|S_{i+1}^{n}=na_{n}-s_{1}^{i})=\frac{\sqrt{n-i}}{\sqrt
{n-i-1}}\pi^{m_{i}}(X_{i+1}=y_{i+1})\big(1+o(1/\sqrt{n})\big),
\]
which together with $(\ref{bayes formula01})$ yields
\begin{align*}
p(X_{1}^{k}=y_{1}^{k}|S_{1}^{n}=na_{n}) &  =\prod_{i=0}^{k-1}\Big(\frac
{\sqrt{n-i}}{\sqrt{n-i-1}}\pi^{m_{i}}(X_{i+1}=y_{i+1})\big(1+o(1/\sqrt
{n})\big)\Big)\\
&=\prod_{i=0}^{k-1}\Big(\pi^{m_i}(X_{i+1}=y_{i+1})%
 \Big)\prod_{i=0}^{k-1}\Big(\frac{\sqrt{ n-i}}{\sqrt{n-i-1}}\Big)\prod_{i=0}^{k-1}\Big(
 1+o\big(\frac{1}{\sqrt{n}}\big)\Big)\\
 &=\Big(
 1+o\big(\frac{1}{\sqrt{n}}\big)\Big)\prod_{i=0}^{k-1}\Big(\pi^{m_i}(X_{i+1}=y_{i+1})%
 \Big),
\end{align*}
which completes the proof.
\end{proof}

\bigskip
In the present case, namely for fixed k, an equivalent statement is (its proof is similar to the proof of Theorem \ref{point conditional density}, we refer to Broniatowski and Cao \cite{BRCAO2012} for more details.)

\begin{thm}
\label{point conditional density e} Under the same notation and hypotheses as in the
previous Theorem, it holds
\[
p_{a_{n}}(y_{1}^{k})=p(X_{1}^{k}=y_{1}^{k}|S_{1}^{n}=na_{n})=g_{a_{n}}%
(y_{1}^{k})\Big(1+o\big(\frac{1}{\sqrt{n}}\big)\Big),
\]
with
\[
g_{a_{n}}(y_{1}^{k})=\prod_{i=1}^{k}\Big(\pi^{a_{n}}(X_{i}=y_{i})\Big).
\]

\end{thm}

\begin{rem}
The above result shows that asymptotically the point condition $\left(
S_{1}^{n}=na_{n}\right)  $ leaves blocks of $k$ of the $X_{i}^{\prime}s$
independent. Obviously this property does not hold for large values of $k,$
close to $n.$ A similar statement holds in the LDP range, conditioning either
on $\left(  S_{1}^{n}=na\right)  $ (see Diaconis and Friedman 1988)), or on
$\left(  S_{1}^{n}\geq na\right)  $; see Csiszar 1984 for a general statement
on asymptotic conditional independence.
\end{rem}

\bigskip

\section{Proofs}\label{Chp2proof}

\subsection{Proofs of Theorem \ref{order of s} and Corollary \ref{3cor1}}\label{3appendix01}

For density functions $p(x)$ defined in $(\ref{densityFunction})$ satisfying
also $h(x)\in\mathfrak{R}$, denote by $\psi(x)$ the reciprocal function of
$h(x)$ and $\sigma^{2}(v)=\big(h^{\prime}(v)\big)^{-1}$, $v\in\mathbb{R}_{+}$.
For brevity, we write $\hat{x},\sigma,l$ instead of $\hat{x}(t),\sigma
\big(\psi(t)\big),l(t)$.

When $t$ is given, $K(x,t)$ attain its maximum at $\hat{x}=\psi(t)$. The
fourth order Taylor expansion of $K(x,t)$ on $x\in[\hat{x}-\sigma l,\hat
{x}+\sigma l]$ yields
\begin{align}
\label{3abeltheorem001}K(x,t)=K(\hat{x},t)-\frac{1}{2}h^{\prime}(\hat
{x})\big(x-\hat{x}\big)^{2} -\frac{1}{6}h^{\prime\prime}(\hat{x}%
)\big(x-\hat{x}\big)^{3}+\varepsilon(x,t),
\end{align}
with some $\theta\in (0,1)$
\begin{align}
\label{3abel01}\varepsilon(x,t)=-\frac{1}{24}h^{^{\prime\prime\prime}}%
\big(\hat{x}+\theta(x-\hat{x})\big)(x-\hat{x})^{4}.
\end{align}

\bigskip
For proving Theorem \ref{order of s} and Corollary \ref{3cor1}, we state firstly the following Lemmas.

\begin{lem}
\label{3lemma00} For $p(x)$ in $(\ref{densityFunction})$, $h(x)\in
\mathfrak{R}$, it holds when $t\rightarrow\infty$,
\begin{align}
\label{3section20211}\frac{|\log\sigma\big(\psi(t)\big)|}{\int_{1}^{t}
\psi(u)du}\longrightarrow0.
\end{align}

\end{lem}

\begin{proof}[Proof] If $h(x)\in R_{\beta}$, by Theorem $(1.5.12)$ of $\cite{Bingham}$,
there exists some slowly varying function such that it holds $\psi(x)\sim
x^{1/\beta}l_{1}(x) $. Hence as $t\rightarrow\infty$ (see
\cite{Feller71}, Chapter 8)
\begin{align}
\label{3section202}\int_{1}^{t} \psi(u)du \sim t^{1+\frac{1}{\beta}}
l_{1}(t).
\end{align}
On the other hand, $h^{\prime}(x)= x^{\beta-1}l(x)\big(\beta+\epsilon
(x)\big)$, thus we have as $x\rightarrow\infty$
\begin{align*}
|\log\sigma(x)| & =\big|\log\big(h^{\prime}(x)\big)^{-\frac{1}{2}%
}\big|=\Big|\frac{1}{2}\big((\beta-1)\log x+\log l(x)+\log(\beta
+\epsilon(x))\big)\Big|\le\frac{1}{2}(\beta+1)\log x,
\end{align*}
set $x=\psi(t)$, then when $t\rightarrow\infty$, it holds $x<2 t^{1/\beta
}l_{1}(t)<t^{1/\beta+1}$, hence we get
\begin{align*}
|\log\sigma\big(\psi(t)\big)|<\frac{(\beta+1)^{2}}{2\beta} \log t,
\end{align*}
which, together with $(\ref{3section202})$, yields $(\ref{3section20211})$.

If $h(x)\in R_{\infty}$, then by definition $\psi(x)\in\widetilde{R_{0}}$ is
slowly varying as $x\rightarrow\infty$, and as $t\rightarrow\infty
$ (see \cite{Feller71}, Chapter 8)
\begin{align}
\label{3section203}\int_{1}^{t} \psi(u)du \sim t\psi(t).
\end{align}
Additionally, we have $h^{\prime}(x)=1/\psi^{\prime}(t)$ with $x=\psi(t)$, it follows
\begin{align*}
|\log\sigma(x)|=\big|\log\big(h^{\prime}(x)\big)^{-\frac{1}{2}}\big|=\frac
{1}{2} |\log\psi^{\prime}(t)|.
\end{align*}
Since $\psi(t)\in\widetilde{R_{0}}$, it holds
\begin{align}
\label{3section204}|\log\sigma\big(\psi(t)\big)| & =\frac{1}{2} |\log
\psi^{\prime}(t)|=\frac{1}{2} \Big|\log\Big(\psi(t)\frac{\epsilon(t)}%
{t}\Big)\Big| \le\log t+\frac{1}{2} |\log\epsilon(t)|\le2\log t,
\end{align}
where last inequality follows from $(\ref{3section1030})$. $(\ref{3section203}%
)$ and $(\ref{3section204})$ imply $(\ref{3section20211})$. This completes the proof.
\end{proof}

\begin{lem}
\label{3lemma0} For $p(x)$ in $(\ref{densityFunction})$, $h\in\mathfrak{R}$,
then for any varying slowly function $l(t)\rightarrow\infty$ as $t\rightarrow
\infty$, it holds
\begin{align}
\label{3section201}\sup_{|x|\le\sigma l}{h^{\prime\prime\prime}(\hat{x}%
+x)}\sigma^{4} l^{4}\longrightarrow0 \qquad as \quad t\rightarrow\infty.
\end{align}

\end{lem}

\begin{proof}[Proof] \textbf{Case 1:} $h\in R_{\beta}$. We have $h(x)=x^{\beta}l_{0}(x) ,
l_{0}(x)\in R_{0},\beta>0$. Then
\begin{align}
\label{mu3 10101}h^{^{\prime\prime}}(x)=\beta(\beta-1)x^{\beta-2}%
l_{0}(x)+2\beta x^{\beta-1}l_{0}^{^{\prime}}(x)+x^{\beta}l_{0}^{^{\prime
\prime}}(x).
\end{align}
and
\begin{align}
\label{mu3 1010}h^{^{\prime\prime\prime}}(x)=\beta(\beta-1)(\beta
-2)x^{\beta-3}l_{0}(x)+3\beta(\beta-1)x^{\beta-2}l_{0}^{^{\prime}}(x)+3\beta
x^{\beta-1}l_{0}^{^{\prime\prime}}(x)+x^{\beta}l_{0}^{^{\prime\prime\prime}%
}(x).
\end{align}
Since $l_0\in R_{0}$, it is easy to obtain
\begin{align}
\label{mu3 10102}l_{0}^{^{\prime}}(x)=\frac{l_{0}(x)}{x}\epsilon(x),\qquad
l_{0}^{^{\prime\prime}}(x)=\frac{l_{0}(x)}{x^{2}}\big(\epsilon^{2}%
(x)+x\epsilon^{\prime}(x)-\epsilon(x)\big),
\end{align}
and
\begin{align*}
l_{0}^{^{\prime\prime\prime}}(x)=\frac{l_{0}(x)}{x^{3}}\big(\epsilon
^{3}(x)+3x\epsilon^{\prime}(x)\epsilon(x)-3\epsilon^{2}(x) -2x\epsilon
^{^{\prime}}(x)+2\epsilon(x)+x^{2}\epsilon^{^{\prime\prime}}(x)\big).
\end{align*}
Under condition $(\ref{3section104})$, there exists some positive constant $Q
$ such that it holds
\begin{align*}
|l_{0}^{^{\prime\prime}}(x)|\le Q\frac{l_{0}(x)}{x^{2}},\qquad|l_{0}%
^{^{\prime\prime\prime}}(x)|\le Q\frac{l_{0}(x)}{x^{3}},
\end{align*}
which, together with $(\ref{mu3 1010})$, yields with some positive constant
$Q_{1}$
\begin{align}
\label{mu3 1011}|h^{^{\prime\prime\prime}}(x)|\le Q_{1} \frac{h(x)}{x^{3}}.
\end{align}
By definition, we have $\sigma^{2}(x)=1/h^{^{\prime}}(x)=x/\big(h(x)(\beta
+\epsilon(x))\big)$, thus it follows
\begin{align}
\label{mu3 1013}\sigma^{2}=\sigma^{2}(\hat{x})=\frac{\hat{x}}{h(\hat{x}%
)(\beta+\epsilon(\hat{x}))}=\frac{\psi(t)}{t(\beta+\epsilon(\psi(t)))}%
=\frac{\psi(t)}{\beta t}\big(1+o(1)\big),
\end{align}
this implies $\sigma l=o(\psi(t))=o(\hat{x})$. Thus we get with
$(\ref{mu3 1011})$
\begin{align}
\label{mu3 1014}\sup_{|x|\le\sigma l}|h^{^{\prime\prime\prime}}(\hat{x}%
+x)|\le\sup_{|x|\le\sigma l} Q_{1} \frac{h(\hat{x}+x)}{(\hat{x}+x)^{3}}\le
Q_{2} \frac{t}{\psi^{3}(t)},
\end{align}
where $Q_{2}$ is some positive constant. Combined with $(\ref{mu3 1013})$, we
obtain
\begin{align*}
\sup_{|x|\le\sigma l}|h^{^{\prime\prime\prime}}(\hat{x}+x)|\sigma^{4} l^{4}
\le Q_{2} \frac{t}{\psi^{3}(t)}\sigma^{4}l^{4}=\frac{Q_{2}l^{4}}{\beta^{2} t
\psi(t)}\longrightarrow0,
\end{align*}
as sought.

\textbf{Case 2:} $h\in R_{\infty}$. Since $\hat{x}=\psi(t)$, we have
$h(\hat{x})=t$. Thus it holds
\begin{align}
\label{mu3 0010}h^{\prime}(\hat{x})=\frac{1}{\psi^{\prime}(t)}\qquad and \quad
h^{\prime\prime}(\hat{x})=-\frac{\psi^{\prime\prime}(t)}{\big(\psi^{\prime
}(t)\big)^{3}},
\end{align}
further we get
\begin{align}
\label{mu3 001}h^{\prime\prime\prime}(\hat{x})=-\frac{\psi^{^{\prime
\prime\prime}}(t)\psi^{^{\prime}}(t)-3\big(\psi^{^{\prime\prime}}(t)\big)^{2}%
}{\big(\psi^{\prime}(t)\big)^{5}}.
\end{align}
Notice if $h(\hat{x})\in R_{\infty}$, then $\psi(t) \in\widetilde{R_{0}}$.
Therefore we obtain
\begin{align}
\label{mu3 002}\psi^{^{\prime}}(t)=\frac{\psi(t)}{t}\epsilon(t),
\end{align}
and
\begin{align}
\label{mu3 003}\psi^{^{\prime\prime}}(t) & =-\frac{\psi(t)}{t^{2}}%
\epsilon(t)\Big(1-\epsilon(t)-\frac{t\epsilon^{^{\prime}}(t)}{\epsilon
(t)}\Big)=-\frac{\psi(t)}{t^{2}}\epsilon(t)\big(1+o(1)\big) \qquad as \quad
t\rightarrow\infty,
\end{align}
where last equality holds from $(\ref{3section103})$. Using
$(\ref{3section103})$ once again, we have also $\psi^{^{\prime\prime\prime}%
}(t)$
\begin{align}
\label{mu3 004}\psi^{^{\prime\prime\prime}}(t) & =\frac{\psi(t)}{t^{3}%
}\epsilon(t)\Big(2+\epsilon^{2}(t)+3t\epsilon^{^{\prime}}(t)-3\epsilon
(t)-\frac{2t\epsilon^{^{\prime}}(t)}{\epsilon(t)}+\frac{t^{2}\epsilon
^{^{\prime\prime}}(t)}{\epsilon(t)}\Big)\nonumber\\
& =\frac{\psi(t)}{t^{3}}\epsilon(t)\big(2+o(1)\big) \qquad as \quad
t\rightarrow\infty.
\end{align}
Put $(\ref{mu3 002})$ $(\ref{mu3 003})$ and $(\ref{mu3 004})$ into
$(\ref{mu3 001})$ we get
\begin{align*}
h^{^{\prime\prime\prime}}(\hat{x})=\frac{t}{\psi^{3}(t)\epsilon^{3}%
(t)}\big(1+o(1)\big)
\end{align*}
Thus by $(\ref{3section1030})$ as $t\rightarrow\infty$
\begin{align}
\label{mu3 005}\sup_{|v|\le t/4}h^{\prime\prime\prime}\big(\psi(t+v)\big) &
=\sup_{|v|\le t/4}\frac{t+v}{\psi^{3}(t+v)\epsilon^{3}(t+v)}%
\big(1+o(1)\big) \le\sup_{|v|\le t/4}\frac{2{(t+v)^{11/8}}}{\psi^{3}(t+v)}\le\frac{4{t}^{11/8}%
}{\psi^{3}(t)},
\end{align}
where last inequality holds from the slowly varying propriety: $\psi
(t+v)\sim\psi(t)$. Using $\sigma=\big(h^{^{\prime}}(\hat{x})\big)^{-1/2}$, it
holds
\begin{align*}
\sup_{|v|\le t/4}{h^{^{\prime\prime\prime}}\big(\psi(t+v)\big)}\sigma^{4}
\le\frac{4{t}^{11/8}}{\psi^{3}(t)}\frac{1}{({h^{^{\prime}}(\hat{x})})^{2}}%
=\frac{4{t}^{11/8}}{\psi^{3}(t)} \frac{\psi^{2}(t)\epsilon^{2}(t)}{t^{2}}%
=\frac{4\epsilon^{2}(t)}{\psi(t)t^{5/8}}\longrightarrow0.
\end{align*}
Hence for any slowly varying function $l(t)\rightarrow\infty$ it holds as $t\rightarrow
\infty$
\begin{align*}
\sup_{|v|\le t/4}{h^{^{\prime\prime\prime}}\big(\psi(t+v)\big)} \sigma^{4}
l^{4} \longrightarrow0.
\end{align*}
Consider $\psi(t)\in\widetilde{R_{0}}$, thus $\psi(t)$ is increasing, we have
the relation
\begin{align*}
\sup_{|v|\le t/4}{h^{^{\prime\prime\prime}}\big(\psi(t+v)\big)}=\sup
_{|\zeta|\le[\zeta_{1},\zeta_{2}]}h^{^{\prime\prime\prime}}(\hat{x}+\zeta),
\end{align*}
where
\begin{align*}
\zeta_{1}=\psi(3t/4)-\hat{x}, \qquad\zeta_{2}=\psi(5t/4)-\hat{x}.
\end{align*}
Hence we have showed
\begin{align*}
\sup_{|\zeta|\le[\zeta_{1},\zeta_{2}]}{h^{^{\prime\prime\prime}}(\hat{x}%
+\zeta)} \sigma^{4} l^{4} \longrightarrow0.
\end{align*}
For completing the proof, it remains to show
\begin{align}
\label{mu3 0060}\sigma l \le\min(|\zeta_{1}|,\zeta_{2}) \qquad as \quad
t\rightarrow\infty.
\end{align}
Perform first order Taylor expansion of $\psi(3t/4)$ at $t$, for some
$\alpha\in[0,1]$, it holds
\begin{align*}
\zeta_{1} & =\psi(3t/4)-\hat{x}=\psi(3t/4)-\psi(t)=-\psi^{^{\prime}%
}\big(t-\alpha{t}/{4}\big) \frac{t}{4}=-\frac{\psi\big(t-\alpha t/{4}%
\big)}{4-\alpha}\epsilon\big(t-\alpha t/{4}\big),
\end{align*}
thus using $(\ref{3section1030})$ and slowly varying propriety of $\psi(t)$ we
get as $t\rightarrow\infty$
\begin{align}
\label{mu3 007}|\zeta_{1}|\ge\frac{\psi\big(t-\alpha t/{4}\big)}{4}%
\epsilon\big(t-\alpha t/{4}\big)\ge\frac{\psi(t)}{5}\epsilon\big(t-\alpha
t/{4}\big)\ge\frac{\psi(t)}{5t^{1/8}}.
\end{align}
On the other hand, we have $\sigma=\big(h^{^{\prime}}(\hat{x})\big)^{-1/2}%
=\big(\psi(t)\epsilon(t)/t\big)^{1/2}$, which, together with $(\ref{mu3 007}%
)$, yields
\begin{align*}
\frac{\sigma}{|\zeta_{1}|}\le5\sqrt{\frac{\epsilon(t)}{\psi(t)\sqrt{t}}%
}\longrightarrow0 \qquad as \quad t\rightarrow\infty,
\end{align*}
which implies for any slowly varying function $l(t)$ it holds $\sigma
l=o(|\zeta_{1}|)$. By the same way, it is easy to show $\sigma l=o(\zeta_{2}%
)$. Hence $(\ref{mu3 0060})$ holds, as sought.
\end{proof}

\begin{lem}
\label{3lemma01} For $p(x)$ in $(\ref{densityFunction})$, $h\in\mathfrak{R}$,
then for any varying slowly function $l(t)\rightarrow\infty$ as $t\rightarrow
\infty$, it holds
\begin{align}
\label{3section201}\sup_{|x|\le\sigma l}\frac{h^{\prime\prime\prime}(\hat
{x}+x)}{h^{\prime\prime}(\hat{x})}\sigma l\longrightarrow0 \qquad as \quad
t\rightarrow\infty,
\end{align}
and
\begin{align}
\label{3section2010}h^{^{\prime\prime}}(\hat{x})\sigma^{3} l \longrightarrow0, \qquad h^{^{\prime\prime}}(\hat{x})\sigma^{4}  \longrightarrow0.
\end{align}

\end{lem}

\begin{proof}[Proof] \textbf{Case 1:} $h\in R_\beta.$ Using $(\ref{mu3 10101})$ and $(\ref{mu3 10102})$, we
get $h^{^{\prime\prime}}(x)=\big(\beta(\beta-1)+o(1)\big)x^{\beta-2} l_{0}(x)$
as $x\rightarrow\infty$, where $l_{0}(x)\in R_{0}$. Hence it holds
\begin{align}
\label{3lem230}h^{^{\prime\prime}}(\hat{x})=\big(\beta(\beta-1)+o(1)\big)\psi
(t)^{\beta-2} l_{0}(\psi(t)),
\end{align}
which, together with $(\ref{mu3 1013} )$ and $(\ref{mu3 1014} )$, yields with
some positive constant $Q_{3}$
\begin{align*}
\sup_{|x|\le\sigma l}\Big|\frac{h^{\prime\prime\prime}(\hat{x}+x)}{h^{\prime\prime}(\hat{x})}\sigma l\Big|\le Q_{3} \frac{t}{\psi^{3}(t)}
\frac{1}{\psi(t)^{\beta-2} l_{0}(\psi(t))}\sqrt{\frac{\psi(t)}{\beta t}}l
=\frac{Q_3}{\sqrt{\beta}} \frac{\sqrt{t}}{\psi(t)^{\beta+1/2}l_0(\psi
(t))}l.
\end{align*}
Notice $\psi(t)\sim t^{1/\beta}l_{1}(t)$ for some slowly varying function
$l_{1}(t)$, then it holds $\sqrt{t}l=o\big(\psi(t)^{\beta+1/2}\big)$. Hence we
get $(\ref{3section201})$.

From $(\ref{mu3 1013} )$ and $(\ref{3lem230})$, we obtain as $t\rightarrow
\infty$
\begin{align}
\label{3hfe1}h^{^{\prime\prime}}(\hat{x})\sigma^{3} l  & =\big(\beta
(\beta-1)+o(1)\big)\psi(t)^{\beta-2} l_{0}(\psi(t))\Big(\frac{\psi(t)}{\beta
t}\Big)^{3/2}l\nonumber\\
& =\big(\beta(\beta-1)+o(1)\big) \frac{\psi(t)^{\beta-1/2}}{\beta^{3/2}
t^{3/2}}l_{0}(\psi(t))l\sim \frac{\beta-1}{\beta^{1/2}}    \frac{l_{1}(t)^{\beta-1/2}}{t^{1/2+1/2\beta}}l_{0}(\psi(t))l
\end{align}
This implies the first formula of $(\ref{3section2010})$ holds.

\textbf{Case 2:} $h\in R_\infty.$ Using $(\ref{mu3 0010})$ and $(\ref{mu3 003})$ we obtain
\begin{align}
\label{mu3 006}h^{\prime\prime}(\hat{x})=-\frac{\psi^{\prime\prime}%
(t)}{\big(\psi^{\prime}(t)\big)^{3}} =\frac{t}{\psi^{2}(t)\epsilon^{2}%
(t)}\big(1+o(1)\big).
\end{align}
Combine $(\ref{mu3 005})$ and $(\ref{mu3 006})$, using $\sigma
=\big(h^{^{\prime}}(\hat{x})\big)^{-1/2}$, we have as $t\rightarrow\infty$
\begin{align*}
\sup_{|v|\le t/4}\frac{h^{^{\prime\prime\prime}}\big(\psi(t+v)\big)}
{h^{^{\prime\prime}}(\hat{x})}\sigma\le\frac{5\epsilon^{2}(t)t^{3/8}}{\psi(t)}%
\frac{1}{\sqrt{h^{^{\prime}}(\hat{x})}}=\frac{5\epsilon(t)^{5/2}}{t^{1/8}\sqrt{\psi
(t)}}\rightarrow0,
\end{align*}
where $\epsilon(t)\rightarrow0$ and $\psi(t)$ varies slowly. Hence for
arbitrarily slowly varying function $l(t)$ it holds as $t\rightarrow\infty$
\begin{align*}
\sup_{|v|\le t/4}\frac{h^{^{\prime\prime\prime}}\big(\psi(t+v)\big)}
{h^{^{\prime\prime}}(\hat{x})}\sigma l \longrightarrow0.
\end{align*}
Define $\zeta_{1},\zeta_{2}$ as in Lemma $\ref{3lemma0}$, we have showed
\begin{align*}
\sup_{|\zeta|\le[\zeta_{1},\zeta_{2}]}\frac{h^{^{\prime\prime\prime}}(\hat
{x}+\zeta)} {h^{^{\prime\prime}}(\hat{x})}\sigma l \longrightarrow0.
\end{align*}
$(\ref{3section201})$ is obtained by using $(\ref{mu3 0060})$. Using
$(\ref{mu3 006})$, for any slowly varying function, it holds
\begin{align*}
h^{^{\prime\prime}}(\hat{x})\sigma^{3} l\sim\frac{l}{\sqrt{\psi(t)\epsilon(t)t}%
}\longrightarrow0.
\end{align*}

 By the same method as proving $h^{^{\prime\prime}}(\hat{x})\sigma^{3}l\rightarrow 0$, it is easy to get for \textbf{Case 1} and \textbf{Case 2}
\begin{align*}
h^{^{\prime\prime}}(\hat{x})\sigma^{4}\longrightarrow 0.
\end{align*}
Hence the proof.
\end{proof}

\begin{lem}
\label{3lemma02} For $p(x)$ in $(\ref{densityFunction})$, $h\in\mathfrak{R}$,
then for any slowly varying function $l(t)\rightarrow\infty$ as $t\rightarrow
\infty$, it holds
\begin{align*}
\sup_{y\in[-l,l]}\frac{|\xi(\sigma y+\hat{x},t)|}{h^{^{\prime\prime}}(\hat
{x})\sigma^{3}}\longrightarrow0,
\end{align*}
where $\xi(x,t)=\varepsilon(x,t)+q(x)$.
\end{lem}

\begin{proof}[Proof] A close look to the proof of Lemma $\ref{3lemma01}$, it is straightforward that $(\ref{3section201})$ can be slightly modified as
\begin{align*}
\sup_{|x|\le\sigma l}\frac{h^{\prime\prime\prime}(\hat
{x}+x)}{h^{\prime\prime}(\hat{x})}\sigma l^4\longrightarrow0 \qquad as \quad
t\rightarrow\infty.
\end{align*}
Hence for $y\in[-l,l]$, by $(\ref{3abel01})$ and Lemma $\ref{3lemma01}$ it holds as $t\rightarrow\infty$
\begin{align}
\label{3lemma0301}\left|\frac{\varepsilon(\sigma y+\hat{x},t)}{h^{^{\prime\prime}%
}(\hat{x})\sigma^{3}}\right|\le\sup_{|x|\le\sigma l}\Big|\frac{h^{\prime\prime\prime
}(\hat{x}+x)}{h^{\prime\prime}(\hat{x})}\Big|\sigma l^4\longrightarrow0.
\end{align}
Under condition $(\ref{densityFunction01})$, set $x=\psi(t)$, we get
\begin{align*}
\sup_{|v-\psi(t)|\le\vartheta\psi(t)}|q(v)|\le\frac{1}{t\sqrt{\psi(t)}}.
\end{align*}

Then we show
\begin{align}\label{fqdq}
\left|\frac{q(\sigma y+\hat{x})}{h^{\prime\prime}(\hat{x})\sigma^{3}}\right| \longrightarrow 0.
\end{align}

\textbf{Case 1:} $h\in R_\beta.$ We have $h(x)=x^\beta l_0(x) , l_0(x)\in R_0,\beta>0$. Hence
\begin{align*}
h^{'}(x)=x^{\beta-1}l_0(x)\big(\beta+\epsilon(x)\big).
\end{align*}
Notice $\psi^{'}(t)=1/h^{'}\big(\psi(t)\big)$, it holds as $t\rightarrow\infty$
\begin{align*}
\frac{\sigma l}{\vartheta\psi(t)}&=\frac{l}{\vartheta\psi(t)\sqrt{h^{'}(\psi(t))}}=\frac{l}
{\vartheta\big(\psi(t)\big)^{(\beta+1)/2}l_0\big(\psi(t)\big)^{1/2}\big(\beta+\epsilon\big(\psi(t)\big)\big)^{1/2}}\longrightarrow 0.
\end{align*}
It follows
\begin{align*}
\sup_{|v-\psi(t)|\le\sigma l}|q(v)|\le\frac{1}{t\sqrt{\psi(t)}}.
\end{align*}
Using the above inequality and by the second line of $(\ref{3hfe1})$, when $y\in[-l,l]$, it holds as $t\rightarrow\infty$
\begin{align*}
&\left|\frac{q(\sigma y+\hat{x})}{h^{\prime\prime}(\hat{x})\sigma^{3}}\right|  \sim \left|q(\sigma
y+\hat{x}) \left(\frac{\beta-1}{\sqrt{\beta}}\frac{\psi(t)^{\beta-1/2}}{
t^{3/2}}l_{0}(\psi(t))\right)^{-1}\right|\\
&\le 2\left|\frac{\sqrt{\beta}}{\beta-1}\right|\sup_{|v-\psi(t)|\le\sigma l}|q(v)| \frac{t^{3/2}}{\psi(t)^{\beta-1/2}l_0(\psi(t))}\le 2\left|\frac{\sqrt{\beta}}{\beta-1}\right|\frac{\sqrt{t}}{\psi(t)^{\beta}l_0(\psi(t))}
\longrightarrow 0,
\end{align*}
where last step holds since $\psi(t)\sim t^{1/\beta}l_1(t)$ for some slowly varying function $l_1$.

\textbf{Case 2:} $h\in R_\infty.$ For any slowly varying function $l(t)$ as $t\rightarrow\infty$
\begin{align*}
\frac{\sigma l}{\vartheta\psi(t)}=\frac{\sqrt{\psi^{^{\prime}}(t)}
l}{\vartheta\psi(t)}=\sqrt{\frac{\epsilon(t)}{t\psi(t)}}\frac{l}{\vartheta
}\longrightarrow0,
\end{align*}
hence
\begin{align*}
\sup_{|v-\psi(t)|\le\sigma l}|q(v)|\le\frac{1}{t\sqrt{\psi(t)}}.
\end{align*}
Using this inequality and $(\ref{mu3 006})$, when $y\in[-l,l]$, it holds as
$t\rightarrow\infty$
\begin{align*}
\left|\frac{q(\sigma y+\hat{x})}{h^{\prime\prime}(\hat{x})\sigma^{3}}\right| &\le 2|q(\sigma
y+\hat{x})| \sqrt{\psi(t)\epsilon(t)t}\le 2\sup_{|v-\psi(t)|\le\sigma l}|q(v)|
\sqrt{\psi(t)\epsilon(t)t}\le\sqrt{\epsilon(t)/t}\rightarrow0.
\end{align*}
$(\ref{fqdq})$, together with $(\ref{3lemma0301})$, completes the proof.
\end{proof}

\begin{lem}
\label{3lemma1} For $p(x)$ belonging to $(\ref{densityFunction})$,
$h(x)\in\mathfrak{R}$, $\alpha\in\mathbb{N}$, denote by
\begin{align*}
\Psi(t,\alpha):=\int_{0}^{\infty}(x-\hat{x})^{\alpha}e^{tx}p(x)dx,
\end{align*}
then there exists some slowly varying function $l(t)$ such that it holds as
$t\rightarrow\infty$
\begin{align*}
\Psi(t,\alpha) & =c\sigma^{\alpha+1} e^{K(\hat{x},t)}T_{1}(t,\alpha
)\big(1+o(1)\big),
\end{align*}
where
\begin{align*}
& T_{1}(t,\alpha)=\int_{-\frac{l^{1/3}}{\sqrt{2}}}^{\frac{l^{1/3}}{\sqrt{2}}}
y^{\alpha}\exp\big(-\frac{y^{2}}{2}\big)dy-\frac{h^{^{\prime\prime}}(\hat
{x})\sigma^{3}}{6}\int_{-\frac{l^{1/3}}{\sqrt{2}}}^{\frac{l^{1/3}}{\sqrt{2}}}
y^{3+\alpha}\exp\big(-\frac{y^{2}}{2}\big)dy.
\end{align*}

\end{lem}

\begin{proof}[Proof] By $(\ref{3abel01})$ and Lemma $\ref{3lemma0}$, for any slowly varying function $l(t)$ it
holds as $t\rightarrow\infty$
\begin{align*}
\sup_{|x-\hat{x}|\le\sigma l} |\varepsilon(x,t)|\rightarrow0.
\end{align*}
Given a slowly varying function $l$ with $l(t)\rightarrow\infty$ and define
the interval $I_{t}$ as follows
\begin{align*}
I_{t}:=\Big(-\frac{l^{1/3}\sigma}{\sqrt{2}},\frac{l^{1/3}\sigma}{\sqrt{2}%
}\Big).
\end{align*}
For large enough $\tau$, when $t\rightarrow\infty$ we can partition
$\mathbb{R}_{+}$ as
\begin{align*}
\mathbb{R}_{+}=\{x:0<x< \tau\}\cup\{x:x \in\hat{x}+I_{t}\}\cup\{x:x\ge\tau,x
\notin\hat{x}+I_{t}\},
\end{align*}
where $\tau$ large enough such that it holds for $x>\tau$
\begin{align}
\label{3abeltheorem0010}p(x)< 2ce^{-g(x)}.
\end{align}

Obviously, for fixed $\tau$, $\{x:0<x< \tau\}\cap\{x:x \in\hat{x}+I_{t}\}={Ø}
$ since for large $t$ we have $\min\big(x:x\in\hat{x}+I_{t}\big)\rightarrow
\infty$ as $t\rightarrow\infty$. Hence it holds
\begin{align}
\label{3section1015}\Psi(t,\alpha) & =\int_{0}^{\tau}(x-\hat{x})^{\alpha
}e^{tx}p(x)dx+\int_{x \in\hat{x}+I_{t}}(x-\hat{x})^{\alpha}e^{tx}%
p(x)dx+\int_{x \notin\hat{x}+I_{t},x>\tau}(x-\hat{x})^{\alpha}e^{tx}%
p(x)dx\nonumber\\
& :=\Psi_{1}(t,\alpha)+\Psi_{2}(t,\alpha)+\Psi_{3}(t,\alpha).
\end{align}
We estimate sequentially $\Psi_{1}(t,\alpha),\Psi_{2}(t,\alpha),\Psi
_{3}(t,\alpha)$ in \textbf{Step 1, Step 2 and Step 3}.

\textbf{Step 1:} Using $(\ref{3abeltheorem0010})$, for $\tau$ large enough, we
have
\begin{align}
\label{3section1002}|\Psi_{1}(t,\alpha)|  & \le\int_{0}^{\tau}|x-\hat
{x}|^{\alpha}e^{tx}p(x)dx\le2c \int_{0}^{\tau}|x-\hat{x}|^{\alpha}%
e^{tx-g(x)}dx \le{2c}t^{-1}\hat{x}^{\alpha
}e^{t\tau}.
\end{align}
We show it holds for $h\in\mathfrak{R}$ as $t\rightarrow\infty$
\begin{align}
\label{3section10020}t^{-1}\hat{x}^{\alpha}e^{t\tau}=o(\sigma^{\alpha
+1}e^{K(\hat{x},t)}h^{^{\prime\prime}}(\hat{x})\sigma^{3}).
\end{align}
$(\ref{3section10020})$ is equivalent to
\begin{align*}
\sigma^{-\alpha-4}t^{-1}\hat{x}^{\alpha}e^{t\tau}\big(h^{^{\prime\prime}}%
(\hat{x})\big)^{-1}=o(e^{K(\hat{x},t)}),
\end{align*}
which is implied by
\begin{align*}
\exp\big(-({\alpha}+4)\log\sigma-\log t +\alpha\log\hat{x}+\tau t-\log
h^{^{\prime\prime}}(\hat{x})\big)=o(e^{K(\hat{x},t)}).
\end{align*}
Since $\hat{x}=\psi(t)$, it holds
\begin{align}
\label{3section10002}K(\hat{x},t)=t\psi(t)-g(\psi(t))=\int_{1}^{t}
\psi(u)du+\psi(1)-g(1),
\end{align}
where the second equality can be easily verified by the change of variable
$u=h(v)$.
By Lemma $(\ref{3lemma00})$, we know $\log\sigma=o(e^{K(\hat{x},t)})$ as
$t\rightarrow\infty$. So it remains to show $t=o(e^{K(\hat{x},t)})$, $\log
\hat{x}=o(e^{K(\hat{x},t)})$ and $\log h^{^{\prime\prime}}(\hat{x}%
)=o(e^{K(\hat{x},t)})$.

If $h(x)\in R_{\beta}$, by Theorem $(1.5.12)$ of $\cite{Bingham}$, it holds
$\psi(x)\sim x^{1/\beta}l_{1}(x) $ with some slowly varying function
$l_{1}(x)$. $(\ref{3section202})$ and $(\ref{3section10002})$ yield
$t=o(e^{K(\hat{x},t)}) $. In addition, $\log\hat{x}=\log\psi(t)\sim\big((1/\beta)
\log t +\log l_1(t)\big)=o(e^{K(\hat{x},t)})$. By $(\ref{3lem230})$, it holds $\log
h^{^{\prime\prime}}(\hat{x})=o(t)$. Thus $(\ref{3section10020})$ holds.

If $h(x)\in R_{\infty}$, $\psi(x)\in\widetilde{R_{0}}$ is slowly varying as
$x\rightarrow\infty$. Therefore, by $(\ref{3section203})$ and
$(\ref{3section10002})$, it holds $t=o(e^{K(\hat{x},t)})$ and $\log\hat
{x}=\log\psi(t)=o(e^{K(\hat{x},t)})$. Using $(\ref{mu3 006})$, we have $\log
h^{^{\prime\prime}}(\hat{x})\sim\log t-2\log\hat{x}-2\log\epsilon(t)$. Under
condition $(\ref{3section1030})$, $\log\epsilon(t)=o(t)$, thus it holds $\log
h^{^{\prime\prime}}(\hat{x})=o(t)$. We get $(\ref{3section10020})$.

$(\ref{3section1002})$ and $(\ref{3section10020})$ yield together
\begin{align}
\label{3section1004}|\Psi_{1}(t,\alpha)|=o(\sigma^{\alpha+1}e^{K(\hat{x}%
,t)}h^{^{\prime\prime}}(\hat{x})\sigma^{3}).
\end{align}

\textbf{Step 2:} Notice $\min\big(x:x\in\hat{x}+I_{t}\big)\rightarrow\infty$
as $t\rightarrow\infty$, which implies both $\varepsilon(x,t)$ and $q(x)$ go to
$0$ when $x\in\hat{x}+I_{t}$. By $(\ref{densityFunction})$ and
$(\ref{3abeltheorem001})$, as $t\rightarrow\infty$
\begin{align*}
\Psi_{2}(t,\alpha) & =\int_{x \in\hat{x}+I_{t}}(x-\hat{x})^{\alpha}%
c\exp\big(K(x,t)+q(x)\big)dx\\
& =\int_{x \in\hat{x}+I_{t}}(x-\hat{x})^{\alpha}c\exp\Big(K(\hat{x}%
,t)-\frac{1}{2}h^{\prime}(\hat{x})\big(x-\hat{x}\big)^{2} -\frac{1}{6}h^{\prime\prime}(\hat{x})\big(x-\hat{x}\big)^{3}%
+\xi(x,t)\Big)dx,
\end{align*}
where $\xi(x,t)=\varepsilon(x,t)+q(x)$. Make the change of variable $y=(x-\hat
{x})/\sigma$, it holds
\begin{align}
\label{3abeltheorem002}\Psi_{2}(t,\alpha)=c\sigma^{\alpha+1} \exp
\big(K(\hat{x},t)\big)\int_{-\frac{l^{1/3}}{\sqrt{2}}}^{\frac{l^{1/3}}%
{\sqrt{2}}} y^{\alpha}\exp\Big(-\frac{y^{2}}{2}-\frac{h^{^{\prime\prime}}%
(\hat{x})\sigma^{3}}{6}y^{3}+\xi(\sigma y+\hat{x},t)\Big)dy.
\end{align}
On $y\in\big(-{l^{1/3}}/\sqrt{2},{l^{1/3}}/\sqrt{2}\big)$, by
$(\ref{3section2010})$, $|h^{^{\prime\prime}}(\hat{x})\sigma^{3}y^{3}%
|\le|h^{^{\prime\prime}}(\hat{x})\sigma^{3} l|\rightarrow0$ as $t\rightarrow
\infty$. Perform the first order Taylor expansion, as $t\rightarrow
\infty$
\begin{align*}
\exp\Big(-\frac{h^{^{\prime\prime}}(\hat{x})\sigma^{3}}{6}y^{3}+\xi(\sigma
y+\hat{x},t)\Big)=1-\frac{h^{^{\prime\prime}}(\hat{x})\sigma^{3}}{6}y^{3}%
+\xi(\sigma y+\hat{x},t)+o_{1}(t,y),
\end{align*}
where
\begin{align}\label{feeqfe}
o_{1}(t,y)=o\Big(-\frac{h^{^{\prime\prime}}(\hat{x})\sigma^{3}}{6}y^{3}%
+\xi(\sigma y+\hat{x},t)\Big).
\end{align}
Hence we obtain
\begin{align*}
&  \int_{-\frac{l^{1/3}}{\sqrt{2}}}^{\frac{l^{1/3}}{\sqrt{2}}} y^{\alpha}%
\exp\Big(-\frac{y^{2}}{2}-\frac{h^{^{\prime\prime}}(\hat{x})\sigma^{3}}%
{6}y^{3}+\xi(\sigma y+\hat{x},t)\Big)dy\\
& =\int_{-\frac{l^{1/3}}{\sqrt{2}}}^{\frac{l^{1/3}}{\sqrt{2}}} \Big(1-\frac
{h^{^{\prime\prime}}(\hat{x})\sigma^{3}}{6}y^{3}+\xi(\sigma y+\hat{x}%
,t)+o_{1}(t,y)\Big)y^{\alpha}\exp\big(-\frac{y^{2}}{2}\big)dy =T_{1}(t,\alpha)+T_{2}(t,\alpha),
\end{align*}
where $T_{1}(t,\alpha)$ and $T_{2}(t,\alpha)$ are defined as follows
\begin{align}
\label{3abeltheorem003} & T_{1}(t,\alpha):=\int_{-\frac{l^{1/3}}{\sqrt{2}}%
}^{\frac{l^{1/3}}{\sqrt{2}}} y^{\alpha}\exp\big(-\frac{y^{2}}{2}%
\big)dy-\frac{h^{^{\prime\prime}}(\hat{x})\sigma^{3}}{6}\int_{-\frac{l^{1/3}%
}{\sqrt{2}}}^{\frac{l^{1/3}}{\sqrt{2}}} y^{3+\alpha}\exp\big(-\frac{y^{2}}%
{2}\big)dy,\nonumber\\
& T_{2}(t,\alpha):=\int_{-\frac{l^{1/3}}{\sqrt{2}}}^{\frac{l^{1/3}}{\sqrt{2}}}
\Big(\xi(\sigma y+\hat{x},t)+o_{1}(t,y)\Big)y^{\alpha}\exp\big(-\frac{y^{2}%
}{2}\big)dy.
\end{align}
For $T_{2}(t,\alpha)$, using $(\ref{feeqfe})$ we have
\begin{align*}
|T_{2}(t,\alpha)|
& \le\sup_{y\in[-l,l]} |\xi(\sigma y+\hat{x},t)|\int_{-\frac{l^{1/3}}{\sqrt
{2}}}^{\frac{l^{1/3}}{\sqrt{2}}} |y|^{\alpha}\exp\big(-\frac{y^{2}}%
{2}\big)dy\\
& \qquad\qquad+\int_{-\frac{l^{1/3}}{\sqrt{2}}}^{\frac{l^{1/3}}{\sqrt{2}}}
\Big(\big|o\big(\frac{h^{^{\prime\prime}}(\hat{x})\sigma^{3}}{6}%
y^{3}\big)\big|+\big|o\big(\xi(\sigma y+\hat{x},t)\big)\big|\Big)|y|^{\alpha
}\exp\big(-\frac{y^{2}}{2}\big)dy\\
& \le2\sup_{y\in[-l,l]} |\xi(\sigma y+\hat{x},t)|\int_{-\frac{l^{1/3}}%
{\sqrt{2}}}^{\frac{l^{1/3}}{\sqrt{2}}} |y|^{\alpha}\exp\big(-\frac{y^{2}}%
{2}\big)dy+|o(h^{^{\prime\prime}}(\hat{x})\sigma^{3})|\int_{-\frac{l^{1/3}%
}{\sqrt{2}}}^{\frac{l^{1/3}}{\sqrt{2}}} |y|^{3+\alpha} \exp\big(-\frac{y^{2}%
}{2}\big)dy\\
& = |o(h^{^{\prime\prime}}(\hat{x})\sigma^{3})|\Big(\int_{-\frac{l^{1/3}%
}{\sqrt{2}}}^{\frac{l^{1/3}}{\sqrt{2}}} |y|^{\alpha}\exp\big(-\frac{y^{2}}%
{2}\big)dy+\int_{-\frac{l^{1/3}}{\sqrt{2}}}^{\frac{l^{1/3}}{\sqrt{2}}}
|y|^{3+\alpha} \exp\big(-\frac{y^{2}}{2}\big)dy\Big),
\end{align*}
where last equality holds from Lemma $\ref{3lemma02}$. Since the integrals in
the last equality are both bounded, it holds as $t\rightarrow\infty$
\begin{align*}
T_{2}(t,\alpha)=o(h^{^{\prime\prime}}(\hat{x})\sigma^{3}).
\end{align*}

When $\alpha$ is even, the second term of $T_{1}(t,\alpha)$ vanishes. When
$\alpha$ is odd, the first term of $T_{1}(t,\alpha)$ vanishes. $h^{^{\prime\prime}}(\hat
{x})\sigma^{3}\rightarrow 0$ by $(\ref{3section2010})$, thus
$T_{1}(t,\alpha)$ is at least the same order than $h^{^{\prime\prime}}(\hat
{x})\sigma^{3}$. It follows as $t\rightarrow\infty$
\begin{align}
\label{3abeltheorem004}T_{2}(t,\alpha)=o(T_{1}(t,\alpha)).
\end{align}
Using $(\ref{3abeltheorem002})$, $( \ref{3abeltheorem003})$ and
$(\ref{3abeltheorem004})$ we get
\begin{align}
\label{3section1014}\Psi_{2}(t,\alpha) & =c\sigma^{\alpha+1} \exp
\big(K(\hat{x},t)\big)T_{1}(t,\alpha)\big(1+o(1)\big).
\end{align}

\textbf{Step 3:} Given $h\in\mathfrak{R}$, for any $t$, $K(x,t)$ as a function
of $x$ ($x>\tau$) is concave since
\begin{align*}
K^{\prime\prime}(x,t)=-h^{\prime}(x)<0.
\end{align*}
Thus we get for $x\notin\hat{x}+I_{t}$ and $x>\tau$
\begin{align}
\label{3abeltheorem005}K(x,t)-K(\hat{x},t)\le\frac{K(\hat{x}+\frac
{l^{1/3}\sigma}{\sqrt{2}}sgn(x-\hat{x}),t)-K(\hat{x},t)}{\frac{l^{1/3}\sigma
}{\sqrt{2}}sgn(x-\hat{x})}(x-\hat{x}),
\end{align}
where
\begin{align*}
sgn(x-\hat{x})=
\begin{cases}
1 \qquad\;\;\; if \quad x\ge\hat{x},\\
-1 \qquad if \quad x<\hat{x}.
\end{cases}
\end{align*}
Using $(\ref{3abeltheorem001})$, we get
\begin{align*}
{K(\hat{x}+\frac{l^{1/3}\sigma}{\sqrt{2}}sgn(x-\hat{x}),t)-K(\hat{x},t)}
\le-\frac{1}{8}h^{\prime}(\hat{x})l^{2/3}\sigma^{2}=-\frac{1}{8}l^{2/3},
\end{align*}
which, combined with $(\ref{3abeltheorem005})$, yields
\begin{align*}
K(x,t)-K(\hat{x},t)\le-\frac{\sqrt{2}}{8}l^{1/3}\sigma^{-1}|x-\hat{x}|.
\end{align*}
\newline We obtain
\begin{align*}
|\Psi_{3}(t,\alpha)| & \le 2c\int_{x \notin\hat{x}+I_{t},x>\tau}|x-\hat
{x}|^{\alpha}\exp\big(K(x,t)\big)dx\\
& \le 2 ce^{K(\hat{x},t)}\int_{|x- \hat{x}|>\frac{l^{1/3}\sigma}{\sqrt{2}}}|x-\hat{x}|^{\alpha}\exp\big(-\frac{\sqrt{2}}{8}l^{1/3}\sigma^{-1}|x-\hat{x}|\big)dx\\
& =2c e^{K(\hat{x},t)}\sigma^{\alpha+1}\int_{|y|>\frac{l^{1/3}}{\sqrt{2}}}|y|^{\alpha}\exp\big(-\frac{\sqrt{2}}{8}l^{1/3}|y|\big)dy\\
& =2c e^{K(\hat{x},t)}\sigma^{\alpha+1} \Big
(2e^{-l^{2/3}/8}\big(1+o(1)\big)\Big),
\end{align*}
where last equality holds when $l\rightarrow\infty$ (see e.g. Theorem 4.12.10
of \cite{Bingham}). With $(\ref{3section1014})$, we obtain
\begin{align*}
\Big|\frac{\Psi_{3}(t,\alpha)}{\Psi_{2}(t,\alpha)}\Big|\le\frac{8e^{-l^{2/3}%
/8}}{|T_{1}(t,\alpha)|}.
\end{align*}
In \textbf{Step 2}, we know $T_{1}(t,\alpha)$ has at least the order
$h^{^{\prime\prime}}(\hat{x})\sigma^{3}$. Hence there exists some positive
constant $Q$ and some slowly varying function $l_2$ with $l_{2}(t)\rightarrow\infty$ such that it holds as
$t\rightarrow\infty$
\begin{align*}
\Big|\frac{\Psi_{3}(t,\alpha)}{\Psi_{2}(t,\alpha)}\Big| & \le\frac
{Qe^{-l_{2}^{2/3}/8}}{h^{^{\prime\prime}}(\hat{x})\sigma^{3}}.
\end{align*}
For example, we can take $l_{2}(t)=(\log t)^{3}$.

If $h\in R_{\beta}$, one close look to $(\ref{3hfe1})$, it is easy to know $h^{^{\prime
\prime}}(\hat{x})\sigma^{3} \ge{1}/{t^{1+1/(2\beta)}}$, with the choice of $l_2$ as above, we have
\begin{align*}
\Big|\frac{\Psi_{3}(t,\alpha)}{\Psi_{2}(t,\alpha)}\Big| & \le Q\exp
\big(-l_{2}^{2/3}/8+(1+1/(2\beta))\log t\big)\longrightarrow0.
\end{align*}

If $h\in R_{\infty}$, using $(\ref{mu3 006})$, then it holds as $t\rightarrow
\infty$
\begin{align}
\label{3section1016}\Big|\frac{\Psi_{3}(t,\alpha)}{\Psi_{2}(t,\alpha)}\Big| &
\le2Q\exp\big(-l_{2}^{2/3}/8+\log\sqrt{t\psi(t)\epsilon(t)}\big)\nonumber\\
& =2Q\exp\Big(-l_{2}^{2/3}/8+({1}/{2})\big(\log t+\log\psi(t)+\log
\epsilon(t)\big)\Big) \longrightarrow0,
\end{align}
where last step holds since $\log\psi(t)=O(\log t)$. The proof is completed by
combining $(\ref{3section1015})$, $(\ref{3section1004})$, $(\ref{3section1014}%
)$ and $(\ref{3section1016})$.
\end{proof}

\begin{proof}[Proof of Theorem \ref{order of s}]By Lemma $\ref{3lemma1}%
$, if $\alpha=0$, it holds $T_{1}(t,0)\sim\sqrt{2\pi}$ as $t\rightarrow\infty
$, hence for $p(x)$ defined in $(\ref{densityFunction})$, we can approximate
$X$'s moment generating function $\Phi(t)$
\begin{align}
\label{3moment001}\Phi(t)=\int_{0}^{\infty}e^{tx}p(x)dx=c\sqrt{2\pi}\sigma
e^{K(\hat{x},t)}\big(1+o(1)\big).
\end{align}

If $\alpha=1$, it holds as $t\rightarrow\infty$,
\begin{align*}
& T_{1}(t,1)=-\frac{h^{^{\prime\prime}}(\hat{x})\sigma^{3}}{6}\int
_{-\frac{l^{1/3}}{\sqrt{2}}}^{\frac{l^{1/3}}{\sqrt{2}}} y^{4}\exp
\big(-\frac{y^{2}}{2}\big)dy=-\frac{\sqrt{2\pi}h^{^{\prime\prime}}(\hat
{x})\sigma^{3}}{2}\big(1+o(1)\big),
\end{align*}
hence we have with $\Psi(t,\alpha)$ defined in Lemma $\ref{3lemma1}$
\begin{align}
\label{3moment020}\Psi(t,1) & =-c\sqrt{2\pi}\sigma^{2} e^{K(\hat{x},t)}%
\frac{h^{^{\prime\prime}}(\hat{x})\sigma^{3}}{2}\big(1+o(1)\big)=-\Phi
(t)\frac{h^{^{\prime\prime}}(\hat{x})\sigma^{4}}{2}\big(1+o(1)\big),
\end{align}
which, together with the definition of $\Psi(t,\alpha)$, yields
\begin{align}
\label{3moment02}\int_{0}^{\infty}xe^{tx}p(x)dx=\Psi(t,1)+\hat{x}%
\Phi(t)=\Big(\hat{x}-\frac{h^{^{\prime\prime}}(\hat{x})\sigma^{4}}%
{2}\big(1+o(1)\big)\Big)\Phi(t).
\end{align}
Hence we get
\begin{align}
\label{3moment021}m(t)&=\frac{d\log\Phi(t)}{dt} =\frac{\int_{0}^{\infty
}xe^{tx}p(x)dx}{\Phi(t)}=\hat{x}-\frac{h^{^{\prime\prime}}(\hat{x})\sigma^{4}%
}{2}\big(1+o(1)\big).
\end{align}
By $(\ref{3section2010})$, as $t\rightarrow\infty$
\begin{align}
\label{3moment0g21}m(t)\sim \hat{x}=\psi(t).
\end{align}

Set $\alpha=2$, as $t\rightarrow\infty$, it follows
\begin{align}
\label{3moment03}\Psi(t,2) & =c\sigma^{3} e^{K(\hat{x},t)}\int_{-\frac
{l^{1/3}}{\sqrt{2}}}^{\frac{l^{1/3}}{\sqrt{2}}} y^{2} \exp\big(-\frac{y^{2}%
}{2}\big)dy\big(1+o(1)\big) =\sigma^{2}%
\Phi(t)\big(1+o(1)\big).
\end{align}
Using $(\ref{3moment020}),(\ref{3moment021})$ and $(\ref{3moment03})$, we
have
\begin{align*}
& \int_{0}^{\infty}\big(x-m(t)\big)^{2} e^{tx}p(x)dx =\int_{0}^{\infty
}\big(x-\hat{x}+\hat{x}-m(t)\big)^{2} e^{tx}p(x)dx\\
& =\int_{0}^{\infty}\big(x-\hat{x}\big)^{2} e^{tx}p(x)dx +2\big(\hat
{x}-m(t)\big)\int_{0}^{\infty}(x-\hat{x}) e^{tx}p(x)dx+\big(\hat
{x}-m(t)\big)^{2}\Phi(t)\\
& =\Psi(t,2)+2\big(\hat{x}-m(t)\big)\Psi(t,1)+\big(\hat{x}-m(t)\big)^{2}%
\Phi(t)\\
& =\sigma^{2}\Phi(t)\big(1+o(1)\big)-h^{^{\prime\prime}}(\hat{x})\sigma
^{4}\Big(\Phi(t)\frac{h^{^{\prime\prime}}(\hat{x})\sigma^{4}}{2}%
\Big)\big(1+o(1)\big) +\Big(\frac{h^{^{\prime\prime}}(\hat{x})\sigma^{4}}%
{2}\Big)^{2}\Phi(t)\big(1+o(1)\big)\\
&=\sigma^{2}\Phi(t)\big(1+o(1)\big)-\frac{(h^{^{\prime\prime}}(\hat{x})\sigma^{3})^{2}}%
{4}\sigma^2\Phi(t)\big(1+o(1)\big)=\sigma^{2}\Phi(t)\big(1+o(1)\big),
\end{align*}
where last equality holds since $h^{^{\prime\prime}}(\hat{x})\sigma^{3}$ goes to $0$ by $(\ref{3section2010})$, thus as $t\rightarrow\infty$
\begin{align}
\label{3moment02012}s^{2}(t)&=\frac{d^{2}\log\Phi(t)}{dt^{2}} =\frac{\int
_{0}^{\infty}\big(x-m(t)\big)^{2}e^{tx}p(x)dx}{\Phi(t)}
\sim \sigma^{2}=\psi^\prime(t).
\end{align}

Set $\alpha=3$, the first term of $T_{1}(t,3)$ vanishes, we obtain as
$t\rightarrow\infty$
\begin{align}
\label{3moment0202}\Psi(t,3) & =-c\sqrt{2\pi}\sigma^{4} e^{K(\hat{x},t)}%
\frac{h^{^{\prime\prime}}(\hat{x})\sigma^{3}}{6}\int_{-\frac{l^{1/3}}{\sqrt
{2}}}^{\frac{l^{1/3}}{\sqrt{2}}} \frac{1}{\sqrt{2\pi}}y^{6}\exp\big(-\frac
{y^{2}}{2}\big)dy\nonumber\\
& =-cM_{6}\sqrt{2\pi}e^{K(\hat{x},t)}\frac{h^{^{\prime\prime}}(\hat{x}%
)\sigma^{7}}{6}\big(1+o(1)\big)=-M_{6}\frac{h^{^{\prime\prime}}(\hat{x}%
)\sigma^{6}}{6}\Phi(t)\big(1+o(1)\big),
\end{align}
where $M_{6}$ denotes the sixth order moment of standard normal distribution.
Using $(\ref{3moment020}),(\ref{3moment021})$, $(\ref{3moment03})$ and
$(\ref{3moment0202})$, we have as $t\rightarrow\infty$
\begin{align*}
& \int_{0}^{\infty}\big(x-m(t)\big)^{3} e^{tx}p(x)dx =\int_{0}^{\infty
}\big(x-\hat{x}+\hat{x}-m(t)\big)^{3} e^{tx}p(x)dx\\
& =\int_{0}^{\infty}\Big((x-\hat{x})^{3}+3(x-\hat{x})^{2}\big(\hat
{x}-m(t)\big)+3(x-\hat{x})\big(\hat{x}-m(t)\big)^{2} +\big(\hat{x}%
-m(t)\big)^{3}\Big)e^{tx}p(x)dx\\
& =\Psi(t,3)+3\big(\hat{x}-m(t)\big)\Psi(t,2)+3\big(\hat{x}-m(t)\big)^{2}%
\Psi(t,1)+\big(\hat{x}-m(t)\big)^{3}\Phi(t)\\
& =-M_{6}\frac{h^{^{\prime\prime}}(\hat{x})\sigma^{6}}{6}\Phi
(t)\big(1+o(1)\big)+(3/2)h^{^{\prime\prime}}(\hat{x})\sigma^{4}(\sigma^{2}%
\Phi(t))\big(1+o(1)\big)\\
& \qquad-3\Big(\frac{h^{^{\prime\prime}}(\hat{x})\sigma^{4}}{2}\Big)^{2}%
\Phi(t)\frac{h^{^{\prime\prime}}(\hat{x})\sigma^{4}}{2}%
\big(1+o(1)\big)+\Big(\frac{h^{^{\prime\prime}}(\hat{x})\sigma^{4}}%
{2}\Big)^{3}\Phi(t)\big(1+o(1)\big)\\
& =\frac{9-M_{6}}{6}h^{^{''}}(\hat{x})\sigma^{6}\Phi(t)\big(1+o(1)\big)
-h^{^{''}}(\hat{x})\sigma^6\Phi(t)\frac
{(h^{^{''}}(\hat{x})\sigma^{3})^{2}}{4}\big(1+o(1)\big)\\
&=\frac{9-M_{6}}{6}h^{^{''}}(\hat{x})\sigma^{6}\Phi(t)\big(1+o(1)\big),
\end{align*}
where last equality holds since $h^{^{''}}(\hat{x})\sigma^{3}\rightarrow 0$ by $(\ref{3section2010})$. Hence we get as $t\rightarrow\infty$
\begin{align}
\label{3moment02022}
\mu_{3}(t)&=\frac{d^{3}\log\Phi(t)}{dt^{3}} =\frac
{\int_{0}^{\infty}\big(x-m(t)\big)^{3}e^{tx}p(x)dx}{\Phi(t)}\nonumber\\
&\sim\frac
{9-M_{6}}{6}h^{^{''}}(\hat{x})\sigma^{6}
=-\frac{9-M_{6}}{6}\frac{\psi^{''}(t)}{\psi^{'}(t)^3}\psi^\prime(t)^3
=\frac{M_{6}-9}{6}\psi^{''}(t).
\end{align}

The proof is completed by combining $(\ref{3moment0g21})$
$(\ref{3moment02012})$ with $(\ref{3moment02022})$.

\end{proof}

\begin{proof}[Proof of Corollary \ref{3cor1}]
The proof is immediate by $(\ref{3section2010})$ of Lemma $\ref{3lemma01}$, from which we get $h^{^{\prime\prime}}(\hat{x})\sigma^{3}\rightarrow 0$ since $l(t)\rightarrow \infty$ as $t\rightarrow\infty.$
By $(\ref{3moment02012})$ and $(\ref{3moment02022})$, it holds as
$t\rightarrow\infty$
\begin{align}
\label{3appen03}\frac{\mu_{3}}{s^{3}}\sim\frac{9-M_{6}}{6}h^{^{\prime\prime}%
}(\hat{x})\sigma^{3}\longrightarrow 0.
\end{align}

\end{proof}

\subsection{Proof of Theorem \ref{3theorem1}} \label{3appendix02}

\begin{proof}[Proof]
\textbf{Step 1:} Denote by
\[
G(x):=\rho_{n}(x)-\phi(x)-\frac{\mu_{3}}{6\sqrt{n}s^{3}}\big(x^{3}%
-3x\big)\phi(x).
\]
Let $\varphi^{a_{n}}(\tau)$ be the characteristic function (c.f) of $\bar{\pi
}^{a_{n}};$ the c.f of $\rho_{n}$ is $\big(\varphi^{a_{n}}(\tau/\sqrt
{n})\big)^{n}$. Hence it holds by Fourier inversion theorem
\[
G(x)=\frac{1}{2\pi}\int_{-\infty}^{\infty}e^{-i\tau x}\Big(\big(\varphi
^{a_{n}}(\tau/\sqrt{n})\big)^{n}-e^{-\frac{1}{2}\tau^{2}}-\frac{\mu_{3}%
}{6\sqrt{n}s^{3}}(i\tau)^{3}e^{-\frac{1}{2}\tau^{2}}\Big)d\tau.
\]
We obtain
\begin{align}\label{fellelem1}
G(x)  \leq\frac{1}{2\pi}\int_{-\infty}^{\infty}\Big|\big(\varphi^{a_{n}}%
(\tau/\sqrt{n})\big)^{n}-e^{-\frac{1}{2}\tau^{2}}-\frac{\mu_{3}}{6\sqrt
{n}s^{3}}(i\tau)^{3}e^{-\frac{1}{2}\tau^{2}}\Big|d\tau.
\end{align}

\textbf{Step 2:} In this step we show that the characteristic function
$\varphi^{a_{n}}$ of $\bar{\pi}^{a_{n}}(x)$ satisfies
\begin{align}\label{fellerLEm2}
\sup_{a_{n}\in\mathbb{R}^{+}}\int|\varphi^{a_{n}}(\tau)|^{2}d\tau<\infty\qquad
and\quad\sup_{a_{n}\in\mathbb{R}^{+},|\tau|\geq\varrho>0}|\varphi^{a_{n}%
}(\tau)|<1,
\end{align}
for any positive $\varrho$ .

It is easy to verify that $r$-order ($r\geq1$) moment $\mu^{r}$ of $\pi
^{a_{n}}(x)$ satisfies
\[
\mu^{r}(t)=\frac{d^{r}\log\Phi(t)}{dt^{r}}\quad with\;t=m^{\leftarrow}(a_{n}),
\]
By Parseval identity
\begin{align}\label{3the13}
\int|\varphi^{a_{n}}(\tau)|^{2}d\tau=2\pi\int(\bar{\pi}^{a_{n}}(x))^{2}%
dx\leq2\pi\sup_{x\in\mathbb{R}}\bar{\pi}^{a_{n}}(x).
\end{align}
For the density function $p(x)$ in $(\ref{densityFunction})$, Theorem 5.4 of
Juszczak and Nagaev $\cite{Nagaev}$ states that the normalized conjugate density of $p(x)$,
namely, $\bar{\pi}^{a_{n}}(x)$ has the propriety
\[
\lim_{a_{n}\rightarrow\infty}\sup_{x\in\mathbb{R}}|\bar{\pi}^{a_{n}}%
(x)-\phi(x)|=0.
\]
Thus for arbitrary positive $\delta$, there exists some positive constant $M $
such that it holds
\[
\sup_{a_{n}\geq M}\sup_{x\in\mathbb{R}}|\bar{\pi}^{a_{n}}(x)-\phi
(x)|\leq\delta,
\]
which entails that $\sup_{a_{n}\geq M}\sup_{x\in\mathbb{R}}\bar{\pi}^{a_{n}%
}(x)<\infty$. When $a_{n}<M$, $\sup_{a_{n}<M}\sup_{x\in\mathbb{R}}\bar{\pi
}^{a_{n}}(x)<\infty;$ hence we have
\[
\sup_{a_{n}\in\mathbb{R}^{+}}\sup_{x\in\mathbb{R}}\bar{\pi}^{a_{n}}(x)<\infty,
\]
which, together with $(\ref{3the13})$, gives the first inequality of $(\ref{fellerLEm2})$. Furthermore,
$\varphi^{a_{n}}(\tau)$ is not periodic, hence the second inequality of
$(\ref{fellerLEm2})$ holds from Lemma $4$ (Chapter $15$, section $1$) of
$\cite{Feller71}$.

\textbf{Step 3: } We complete the proof by showing that for $n$ large enough
\begin{equation}
\int_{-\infty}^{\infty}\Big|\big(\varphi^{a_{n}}(\tau/\sqrt{n})\big)^{n}%
-e^{-\frac{1}{2}\tau^{2}}-\frac{\mu_{3}}{6\sqrt{n}s^{3}}(i\tau)^{3}%
e^{-\frac{1}{2}\tau^{2}}\Big|d\tau=o\Big(\frac{1}{\sqrt{n}}%
\Big).\label{3theo100}%
\end{equation}

For arbitrarily positive sequence $a_{n}$ we have
\[
\sup_{a_{n}\in\mathbb{R}^{+}}\Big|\varphi^{a_{n}}(\tau)\Big|=\sup_{a_{n}%
\in\mathbb{R}^{+}}\Big|\int_{-\infty}^{\infty}e^{i\tau x}\bar{\pi}^{a_{n}%
}(x)dx\Big|\leq\sup_{a_{n}\in\mathbb{R}^{+}}\int_{-\infty}^{\infty
}\Big|e^{i\tau x}\bar{\pi}^{a_{n}}(x)\Big|dx=1.
\]
In addition, $\pi^{a_{n}}(x)$ is integrable, by Riemann-Lebesgue theorem, it
holds when $|\tau|\rightarrow\infty$
\[
\sup_{a_{n}\in\mathbb{R}^{+}}\Big|\varphi^{a_{n}}(\tau)\Big|\longrightarrow0.
\]
Thus for any strictly positive $\omega$, there exists some corresponding
$N_{\omega}$ such that if $|\tau|>\omega$, it holds
\begin{align}\label{3the20}
\sup_{a_{n}\in\mathbb{R}^{+}}\Big|\varphi^{a_{n}}(\tau)\Big|<N_{\omega}<1.
\end{align}
We now turn to (\ref{3theo100}) which is splitted on $|\tau|>\omega\sqrt{n} $
and on $|\tau|\leq\omega\sqrt{n}$ .

It holds
\begin{align}
&  \sqrt{n}\int_{|\tau|>\omega\sqrt{n}}\Big|\big(\varphi^{a_{n}}(\tau/\sqrt
{n})\big)^{n}-e^{-\frac{1}{2}\tau^{2}}-\frac{\mu_{3}}{6\sqrt{n}s^{3}}%
(i\tau)^{3}e^{-\frac{1}{2}\tau^{2}}\Big|d\tau\nonumber\label{3the21}\\
&  \leq\sqrt{n}N_{\omega}^{n-2}\int_{|\tau|>\omega\sqrt{n}}\Big|\big(\varphi
^{a_{n}}(\tau/\sqrt{n})\big)\Big|^{2}d\tau+\sqrt{n}\int_{|\tau|>\omega\sqrt
{n}}e^{-\frac{1}{2}\tau^{2}}\Big(1+\Big|\frac{\mu_{3}\tau^{3}}{6\sqrt{n}s^{3}%
}\Big|\Big)d\tau.
\end{align}
where the first term of the last line tends to $0$ for $n$ large enough,
since
\begin{align}
&  \sqrt{n}N_{\omega}^{n-2}\int_{|\tau|>\omega\sqrt{n}}\Big|\big(\varphi
^{a_{n}}(\tau/\sqrt{n})\big)\Big|^{2}d\tau\nonumber\label{3the22}\\
&  =\exp\Big(\frac{1}{2}\log n+(n-2)\log N_{\omega}+\log\int_{|\tau
|>\omega\sqrt{n}}\Big|\big(\varphi^{a_{n}}(\tau/\sqrt{n})\big)\Big|^{2}%
d\tau\Big)\longrightarrow0,
\end{align}
where the last step holds from Lemma $(\ref{fellerLEm2})$ and $(\ref{3the20})$. As for
the second term of $(\ref{3the21})$, by Corollary $(\ref{3cor1})$, for $n$ large enough, we have $|\mu_{3}/s^{3}|\rightarrow0$. Hence it holds
for $n$ large enough
\begin{align}
&  \sqrt{n}\int_{|\tau|>\omega\sqrt{n}}e^{-\frac{1}{2}\tau^{2}}%
\Big(1+\Big|\frac{\mu_{3}\tau^{3}}{6\sqrt{n}s^{3}}\Big|\Big)d\tau
\nonumber\label{3the23}\\
&  \leq\sqrt{n}\int_{|\tau|>\omega\sqrt{n}}e^{-\frac{1}{2}\tau^{2}}|\tau
|^{3}d\tau=\sqrt{n}\int_{|\tau|>\omega\sqrt{n}}\exp\Big\{-\frac{1}{2}\tau
^{2}+3\log|\tau|\Big\}d\tau\nonumber\\
&  =2\sqrt{n}\exp\big(-\omega^{2}n/2+o(\omega^{2}n/2)\big)\longrightarrow0,
\end{align}
where the second equality holds from, for example, Chapter $4$ of
$\cite{Bingham}$. $(\ref{3the21})$, $(\ref{3the22})$ and $(\ref{3the23})$
implicate that, for $n$ large enough
\begin{align}\label{3theo10}
\int_{|\tau|>\omega\sqrt{n}}\Big|\big(\varphi^{a_{n}}(\tau/\sqrt{n}%
)\big)^{n}-e^{-\frac{1}{2}\tau^{2}}-\frac{\mu_{3}}{6\sqrt{n}s^{3}}(i\tau
)^{3}e^{-\frac{1}{2}\tau^{2}}\Big|d\tau=o\Big(\frac{1}{\sqrt{n}}\Big).
\end{align}

If $|\tau|\leq\omega\sqrt{n}$, it holds
\begin{align}
&  \int_{|\tau|\leq\omega\sqrt{n}}\Big|\big(\varphi^{a_{n}}(\tau/\sqrt
{n})\big)^{n}-e^{-\frac{1}{2}\tau^{2}}-\frac{\mu_{3}}{6\sqrt{n}s^{3}}%
(i\tau)^{3}e^{-\frac{1}{2}\tau^{2}}\Big|d\tau\nonumber\label{3the24}\\
&  =\int_{|\tau|\leq\omega\sqrt{n}}e^{-\frac{1}{2}\tau^{2}}\Big|\exp
\Big\{n\log\varphi^{a_{n}}(\tau/\sqrt{n})+{\frac{1}{2}\tau^{2}}\Big\}-1-\frac
{\mu_{3}}{6\sqrt{n}s^{3}}(i\tau)^{3}\Big|d\tau.
\end{align}
The integrand in the last display is bounded through
\begin{align}\label{3the240}
|e^{\alpha}-1-\beta|=|(e^{\alpha}-e^{\beta})+(e^{\beta}-1-\beta)|\leq
(|\alpha-\beta|+\frac{1}{2}\beta^{2})e^{\gamma},
\end{align}
where $\gamma\geq\max(|\alpha|,|\beta|);$ this inequality follows replacing
$e^{\alpha},e^{\beta}$ by their power series, for real or complex
$\alpha,\beta$. Denote by
\[
\gamma(\tau)=\log\varphi^{a_{n}}(\tau)+{\frac{1}{2}\tau^{2}}.
\]
Since $\gamma^{\prime}(0)=\gamma^{\prime\prime}(0)=0$, the third order Taylor
expansion of $\gamma(\tau)$ at $\tau=0$ yields
\[
\gamma(\tau)=\gamma(0)+\gamma^{\prime}(0)\tau+\frac{1}{2}\gamma^{\prime\prime
}(0)\tau^{2}+\frac{1}{6}\gamma^{\prime\prime\prime}(\xi)\tau^{3}=\frac{1}%
{6}\gamma^{\prime\prime\prime}(\xi)\tau^{3},
\]
where $0<\xi<\tau$. Hence it holds
\[
\Big|\gamma(\tau)-\frac{\mu_{3}}{6s^{3}}(i\tau)^{3}\Big|=\Big|\gamma
^{\prime\prime\prime}(\xi)-\frac{\mu_{3}}{s^{3}}i^{3}\Big|\frac{|\tau
|^{3}}{6}.
\]
Here $\gamma^{\prime\prime\prime}$ is continuous; thus we can choose $\omega$
small enough such that $|\gamma^{\prime\prime\prime}(\xi)|<\rho$ for
$|\tau|<\omega$. Meanwhile, for $n$ large enough, according to Corollary
$(\ref{3cor1})$ , we have $|\mu_{3}/s^{3}|\rightarrow0$. Hence it holds for
$n$ large enough
\begin{align}\label{3the25}
\Big|\gamma(\tau)-\frac{\mu_{3}}{6s^{3}}(i\tau)^{3}\Big|\leq\Big(|\gamma
^{\prime\prime\prime}(\xi)|+\rho\Big)\frac{|\tau|^{3}}{6}<\rho|\tau|^{3}.
\end{align}
Choose $\omega$ small enough, such that for $n$ large enough it holds for
$|\tau|<\omega$
\[
\Big|\frac{\mu_{3}}{6s^{3}}(i\tau)^{3}\Big|\leq\frac{1}{4}\tau^{2}%
,\qquad|\gamma(\tau)|\leq\frac{1}{4}\tau^{2}.
\]
For this choice of $\omega$, when $|\tau|<\omega$ we have
\begin{align}\label{3the250}
\max\Big(\Big|\frac{\mu_{3}}{6s^{3}}(i\tau)^{3}\Big|,|\gamma(\tau
)|\Big)\leq\frac{1}{4}\tau^{2}.
\end{align}
Replacing $\tau$ by $\tau/\sqrt{n}$, it holds for $|\tau|<\omega\sqrt{n}$
\begin{align}\label{3the251}
&  \Big|n\log\varphi^{a_{n}}(\tau/\sqrt{n})+{\frac{1}{2}\tau^{2}}-\frac
{\mu_{3}}{6\sqrt{n}s^{3}}(i\tau)^{3}\Big| =n\Big|\gamma\Big(\frac{\tau}{\sqrt{n}}\Big)-\frac{\mu_{3}}{6s^{3}%
}\Big(\frac{i\tau}{\sqrt{n}}\Big)^{3}\Big|<\frac{\rho|\tau|^{3}}{\sqrt{n}},
\end{align}
where the last inequality holds from $(\ref{3the25})$. In a similar way, with
$(\ref{3the250})$, it also holds for $|\tau|<\omega\sqrt{n}$
\begin{align}
&  \max\Big(\Big|n\log\varphi^{a_{n}}(\tau/\sqrt{n})+{\frac{1}{2}\tau^{2}%
}\Big|,\Big|\frac{\mu_{3}}{6\sqrt{n}s^{3}}(i\tau)^{3}%
\Big|\Big)\nonumber\label{3the252}\\
&  =n\max\Big(\Big|\gamma\Big(\frac{\tau}{\sqrt{n}}\Big)\Big|,\Big|\frac
{\mu_{3}}{6s^{3}}\Big(\frac{i\tau}{\sqrt{n}}\Big)^{3}\Big|\Big)\leq\frac{1}%
{4}\tau^{2}.
\end{align}

Apply $(\ref{3the240})$ to estimate the integrand of last line of
$(\ref{3the24})$, with the choice of $\omega$ in $(\ref{3the25})$ and
$(\ref{3the250})$, using $(\ref{3the251})$ and $(\ref{3the252})$ we have for
$|\tau|<\omega\sqrt{n}$
\begin{align*}
&  \Big|\exp\Big\{n\log\varphi^{a_{n}}(\tau/\sqrt{n})+{\frac{1}{2}\tau^{2}%
}\Big\}-1-\frac{\mu_{3}}{6\sqrt{n}s^{3}}(i\tau)^{3}\Big|\\
&  \leq\Big(\Big|n\log\varphi^{a_{n}}(\tau/\sqrt{n})+{\frac{1}{2}\tau^{2}%
}-\frac{\mu_{3}}{6\sqrt{n}s^{3}}(i\tau)^{3}\Big|+\frac{1}{2}\Big|\frac{\mu
_{3}}{6\sqrt{n}s^{3}}(i\tau)^{3}\Big|^{2}\Big)\\
&  \qquad\times\exp\Big[\max\Big(\Big|n\log\varphi^{a_{n}}(\tau/\sqrt
{n})+{\frac{1}{2}\tau^{2}}\Big|,\Big|\frac{\mu_{3}}{6\sqrt{n}s^{3}}(i\tau
)^{3}\Big|\Big)\Big]\\
&  \leq\Big(\frac{\rho|\tau|^{3}}{\sqrt{n}}+\frac{1}{2}\Big|\frac{\mu_{3}%
}{6\sqrt{n}s^{3}}(i\tau)^{3}\Big|^{2}\Big)\exp\Big(\frac{\tau^{2}}{4}\Big)  =\Big(\frac{\rho|\tau|^{3}}{\sqrt{n}}+\frac{\mu_{3}^{2}\tau^{6}}{72ns^{6}%
}\Big)\exp\Big(\frac{\tau^{2}}{4}\Big).
\end{align*}
Use this upper bound to $(\ref{3the24})$, we obtain
\begin{align*}
&  \int_{|\tau|\leq\omega\sqrt{n}}\Big|\big(\varphi^{a_{n}}(\tau/\sqrt
{n})\big)^{n}-e^{-\frac{1}{2}\tau^{2}}-\frac{\mu_{3}}{6\sqrt{n}s^{3}}%
(i\tau)^{3}e^{-\frac{1}{2}\tau^{2}}\Big|d\tau\\
&  \leq\int_{|\tau|\leq\omega\sqrt{n}}\exp\Big(-\frac{\tau^{2}}{4}%
\Big)\Big(\frac{\rho|\tau|^{3}}{\sqrt{n}}+\frac{\mu_{3}^{2}\tau^{6}}{72ns^{6}%
}\Big)d\tau\\
&  =\frac{\rho}{\sqrt{n}}\int_{|\tau|\leq\omega\sqrt{n}}\exp\Big(-\frac
{\tau^{2}}{4}\Big)|\tau|^{3}d\tau+\frac{\mu_{3}^{2}}{72ns^{6}}\int_{|\tau
|\leq\omega\sqrt{n}}\exp\Big(-\frac{\tau^{2}}{4}\Big)\tau^{6}d\tau,
\end{align*}
where both the first integral and the second integral are finite, and $\rho$
is arbitrarily small; additionally, by Corollary $(\ref{3cor1})$, ${\mu
_{3}^{2}}/{s^{6}}\rightarrow0$ when $n$ large enough, hence it holds for $n$ large enough
\begin{align}\label{3theo11}
\int_{|\tau|\leq\omega\sqrt{n}}\Big|\big(\varphi^{a_{n}}(\tau/\sqrt
{n})\big)^{n}-e^{-\frac{1}{2}\tau^{2}}-\frac{\mu_{3}}{6\sqrt{n}s^{3}}%
(i\tau)^{3}e^{-\frac{1}{2}\tau^{2}}\Big|d\tau=o\Big(\frac{1}{\sqrt{n}}\Big).
\end{align}
Now $(\ref{3theo10})$ and $(\ref{3theo11})$ give $(\ref{3theo100})$. Further, using $(\ref{3theo100})$ and $(\ref{fellelem1})$, we obtain
\[
\Big|\rho_n(x)-\phi(x)-\frac{\mu_{3}}{6\sqrt{n}s^{3}}%
\big(x^{3}-3x\big)\phi(x)\Big|=o\Big(\frac{1}{\sqrt{n}}\Big),
\]
which concludes the proof.
\end{proof}

\subsection{Proof of Lemma \ref{3lemma z}} \label{3appendix03}

\begin{proof}[Proof] When $n\rightarrow\infty$, it holds
\[
z_{i}\sim{m_{i}}/{s_{i}\sqrt{n-i-1}}\sim m_{i}/(s_{i}\sqrt{n}).
\]
From Theorem $\ref{order of s}$, it holds $m(t)\sim{\psi(t)}$ and
$s(t)\sim\sqrt{\psi^{^{\prime}}(t)}$. Hence we have
\begin{align}\label{3felff}
z_{i}\sim\frac{\psi(t_{i})}{\sqrt{n\psi^{^{\prime}}(t_{i})}}.
\end{align}
By $(\ref{3lfd01})$, $m_i\sim m(t)$ as $n\rightarrow\infty$. Then
\begin{align*}
m_i \sim \psi(t)=a_n.
\end{align*}
In addition, $m_i \sim \psi(t_i)$ by Theorem $\ref{order of s}$, this implies
\begin{align}\label{3fel}
\psi(t_i) \sim \psi(t).
\end{align}

\textbf{Case 1:} if $h(x)\in R_\beta$.
We have $h(x)=x^\beta l_0(x) , l_0(x)\in R_0,\beta>0$. Hence
\begin{align*}
h^{'}(x)=x^{\beta-1}l_0(x)\big(\beta+\epsilon(x)\big),
\end{align*}
set $x=\psi(u)$, we get
\begin{align}\label{3fel0}
h^{'}\big(\psi(u)\big)=\big(\psi(u)\big)^{\beta-1}l_0\big(\psi(u)\big)\big(\beta+\epsilon\big(\psi(u)\big)\big).
\end{align}
Notice $\psi^{'}(u)=1/h^{'}\big(\psi(u)\big)$, combine $(\ref{3fel})$ with $(\ref{3fel0})$, we obtain
\begin{align}\label{3fge01}
\frac{\psi^\prime(t_i)}{\psi^\prime(t)}=\frac{h^{'}\big(\psi(t)\big)}{h^{'}\big(\psi(t_i)\big)}
=\frac{\big(\psi(t)\big)^{\beta-1}l_0\big(\psi(t)\big)\big(\beta+\epsilon\big(\psi(t)\big)\big)}
{\big(\psi(t_i)\big)^{\beta-1}l_0\big(\psi(t_i)\big)\big(\beta+\epsilon\big(\psi(t_i)\big)\big)}\longrightarrow 1,
\end{align}
where we use the slowly varying propriety of $l_0$. Thus it holds
\begin{align*}
\psi^\prime(t_i) \sim \psi^\prime(t),
\end{align*}
which, together with $(\ref{3fel})$, is put into $(\ref{3felff})$ to yield
\begin{align}\label{3f41ge011}
z_i  \sim \frac{\psi(t)}{\sqrt{n\psi^{'}(t)}}.
\end{align}
Hence we have under condition $(\ref{croissance de a})$
\begin{align}\label{3f41ge0111}
z_i^2  \sim \frac{\psi(t)^2}{{n\psi^{'}(t)}}=\frac{\psi(t)^2}{{\sqrt{n}\psi^{'}(t)}}\frac{1}{\sqrt{n}}=o\Big(\frac{1}{\sqrt{n}}\Big),
\end{align}
which implies further $z_{i}\rightarrow0$. Note that the final step is used in
order to relax the strength of the growth condition on $a_{n}.$

\textbf{Case 2:} if $h(x)\in R_{\infty}$. By $(\ref{3lfd01})$, it holds
$m(t_{i})\ge m(t)$ as $n\rightarrow\infty$. Since the function
$t\rightarrow m(t)$ is increasing, we have
\[
t\leq t_{i}.
\]
Notice the function $x\rightarrow \psi(x)$ is also increasing, we get
\[
\psi(t_i)\ge \psi(t).
\]
The function $x\rightarrow\psi^{^{\prime}}(x)$ is decreasing, since
\begin{align}\label{3fqdqg}
\psi^{^{\prime\prime}}(x)=-\frac{\psi(x)}{x^{2}}\epsilon
(x)\big(1+o(1)\big)<0\qquad as\quad x\rightarrow\infty.
\end{align}
Therefore as $n\rightarrow\infty$
\[
\psi^{\prime}(t)\geq\psi^{\prime}(t_{i})>0.
\]
Perform one Taylor expansion of $\psi(t_i)$ for some $\theta_1\in (0,1)$
\begin{align}\label{3qdm}
\psi(t_i)-\psi(t)&=\psi^\prime(t)(t_i-t)+\frac{1}{2}\psi^{''}(t+\theta_1(t_i-t))(t_i-t)^2\nonumber\\
&=\frac{\psi(t)\epsilon(t)}{t}(t_i-t)+\frac{1}{2}\psi^{''}(t+\theta_1(t_i-t))(t_i-t)^2.
\end{align}
By $(\ref{3fel})$
\begin{align*}
\frac{\psi(t_i)-\psi(t)}{\psi(t)}\longrightarrow 0,
\end{align*}
which together with $(\ref{3fqdqg})$ and $(\ref{3qdm})$ yields
\begin{align}\label{3qdm65}
\frac{\epsilon(t)}{t}(t_i-t)\longrightarrow 0.
\end{align}
Perform one Taylor expansion of $\psi^{\prime}(t_i)$ for some $\theta_2\in (0,1)$
\begin{align*}
\psi^\prime(t_i)-\psi^\prime(t)&=\psi^{''}(t)(t_i-t)+\frac{1}{2}\psi^{'''}(t+\theta_2(t_i-t))(t_i-t)^2\nonumber\\
&=-\frac{\psi(t)\epsilon(t)}{t^2}(t_i-t)\big(1+o(1)\big)+\frac{1}{2}\psi^{'''}(t+\theta_2(t_i-t))(t_i-t)^2,
\end{align*}
where the first term goes to $0$ as $n\rightarrow\infty$ by $(\ref{3qdm65})$, and the second term is infinitely small with respect to the first term (see Section \ref{Chp2proof}, e.g. $(\ref{mu3 004})$). Hence
\begin{align*}
\psi^\prime(t_i)\sim\psi^\prime(t).
\end{align*}
The proof is completed by repeating steps $(\ref{3f41ge011})$ and $(\ref{3f41ge0111})$.

\end{proof}

\end{document}